\documentclass{amsart}
\usepackage{graphicx}

\usepackage{amssymb}
 \usepackage{mathtools}
 \usepackage{kpfonts}
\usepackage[T1]{fontenc}
 \usepackage{multicol}
\usepackage{mathrsfs}
\usepackage[usenames,dvipsnames]{xcolor}
\usepackage{multicol}
\usepackage{amsthm}
\usepackage{float}
\usepackage{tikz}
\usepackage{tikz-cd}
\usetikzlibrary{arrows}
\usetikzlibrary{shapes.arrows}
\usetikzlibrary{positioning}
\usetikzlibrary{decorations.pathreplacing,calligraphy}
\usepackage{mathrsfs}
\usepackage{amssymb}
\usepackage{mathtools}
\usepackage{chessfss}
\usepackage{xfrac}

\usepackage{geometry}
\PassOptionsToPackage{obeyspaces}{url}
\usepackage[hyphens]{xurl}
\usepackage{hyperref}
\hypersetup{breaklinks=true,
	colorlinks=true,
	linkcolor=b-violet,
	urlcolor=b-blue,
	citecolor=b-mauve,}
\usepackage{algorithm}
\usepackage{algpseudocode}

\usepackage{subcaption}

\definecolor{1}{rgb}{1,0.2,0.3}
\definecolor{2}{rgb}{0.1,0.3,0.5}
\definecolor{3}{rgb}{1,1,0}
\definecolor{4}{rgb}{255,255,255}
    \definecolor{sangria}{rgb}{0.57, 0.0, 0.04}
	\definecolor{darkblue}{RGB}{11, 11, 69}
	\definecolor{forestgreen}{RGB}{74,129,34}
	\definecolor{strongpink}{RGB}{210,29,129}
	\definecolor{pan-red}{RGB}{254,33,139}
	\definecolor{pan-yellow}{RGB}{210,178,0}
	\definecolor{pan-blue}{RGB}{33,176,254}
    \definecolor{gq-green}{RGB}{74,129,34}
    \definecolor{gq-mauve}{RGB}{181,126,220}
    \definecolor{tr-pink}{RGB}{246,170,183}
    \definecolor{tr-blue}{RGB}{85, 205, 253}
    \definecolor{b-violet}{RGB}{215,0,113}
    \definecolor{b-blue}{RGB}{0,53,170}
    \definecolor{b-mauve}{RGB}{156,78,151}
    \definecolor{le-red}{RGB}{212, 44, 0}
    \definecolor{le-orange}{RGB}{253, 152, 85}
    \definecolor{le-pink}{RGB}{209, 97, 162}
    \definecolor{le-purple}{RGB}{162, 1, 97}
    \definecolor{g-green}{RGB}{7, 141, 112}
    \definecolor{g-blue}{RGB}{80, 73, 204}

\usepackage[backend=biber,hyperref,backref,backrefstyle=none,
firstinits,
block=space]{biblatex}
\renewbibmacro{in:}{} 

\newtheorem{theorem}{Theorem}
\newtheorem{corollary}[theorem]{Corollary}
\newtheorem{lemma}[theorem]{Lemma}
\newtheorem{conjecture}[theorem]{Conjecture}
\newtheorem{proposition}[theorem]{Proposition}
\newtheorem{tool}[theorem]{Tool}
\newcommand{\psat}{\ensuremath{\mathrm{P3SAT}_{3}}}

\theoremstyle{definition}
\newtheorem*{definition}{Definition}

\theoremstyle{remark}
\newtheorem*{remark}{Remark}

\newcommand{\true}{\texttt{true}{}} 
\newcommand{\false}{\texttt{false}{}} 

\newcommand{\elemcube}[4][brown]{
	\draw [fill=#1!30,thin] (#2+1,#3,#4) -- ++(0,1,0) -- ++(0,0,-1) -- ++(0, -1, 0) -- cycle;
	\draw [fill=#1!40, thin] (#2,#3+1,#4) -- ++(1,0,0) -- ++(0,0,-1) -- ++(-1, 0, 0) -- cycle;
	\draw [fill=#1!10, thin] (#2,#3,#4)   -- ++(1,0,0) -- ++(0,1,0)  -- ++(-1, 0, 0) -- cycle;
}
\newcommand{\glasselemcube}[5][brown]{
	\draw [fill=#1!30,thin, opacity=#5] (#2+1,#3,#4) -- ++(0,1,0) -- ++(0,0,-1) -- ++(0, -1, 0) -- cycle;
	\draw [fill=#1!40, thin, opacity =#5] (#2,#3+1,#4) -- ++(1,0,0) -- ++(0,0,-1) -- ++(-1, 0, 0) -- cycle;
	\draw [fill=#1!10, thin, opacity = #5] (#2,#3,#4)   -- ++(1,0,0) -- ++(0,1,0)  -- ++(-1, 0, 0) -- cycle;
}
\newcommand{\colourcubeZface}[4][brown]
{\draw [fill=#1!30, thin] (#2,#3,#4)   -- ++(1,0,0) -- ++(0,1,0)  -- ++(-1, 0, 0) -- cycle;     }
\newcommand{\colourcubeXface}[4][brown]{
\draw [fill=#1!40, thin] (#2,#3+1,#4) -- ++(1,0,0) -- ++(0,0,-1) -- ++(-1, 0, 0) -- cycle;
}

\newcommand{\cubelabel}[4]{
\draw (#1+.5,#2+.5,#3) node {#4};
}

\makeatletter
\DeclareRobustCommand
  \myvdots{\vbox{\baselineskip4\p@ \lineskiplimit\z@
    \hbox{.}\hbox{.}\hbox{.}}}
\makeatother

\begin{document}

\tikzset
{
  x=.23in,
  y=.23in,
}

\title{Complexity of chess domination problems}
\author[A Langlois-Rémillard]{Alexis Langlois-R\'emillard}
\address[
A. Langlois-R\'emillard]{Department of applied mathematics, computer science and statistics, Faculty of Sciences, Ghent University, Ghent, Belgium\\ and
ScaDS.AI Leipzig, Universität Leipzig, Leipzig, Germany\\
now at Hausdorff Center for Mathematics, Bonn, Germany}
\email{alexis.langlois-remillard@tutanota.com}
\author[M Müßig]{Mia M\"u\ss{}ig}
\address[M. Müßig ]{Ludwig Maximilian Universit\"at M\"unchen, Germany and
ScaDS.AI Leipzig, Universität Leipzig, Leipzig, Germany}
\email{nienna@miamuessig.de}
\author[\'E Rold\'an]{\'Erika Rold\'an*}
\address[*\'E. Rold\'an]{ Max Planck Institute for Mathematics in the Sciences, Inselstraße 22, and
ScaDS.AI Leipzig, Universität Leipzig, Leipzig, Germany}
\email{roldan@mis.mpg.de}

\maketitle

\begin{abstract}
We study different domination problems of attacking and non-attacking rooks and queens on polyominoes and polycubes of all dimensions. Our main result proves that maximum independent domination is NP-complete for non-attacking queens and for non-attacking rooks on polycubes of dimension three and higher. We also analyze these problems for polyominoes and convex polyominoes, conjecture the complexity classes, and provide a computer tool for investigation. We have also computed new values for classical queen domination problems on chessboards (square polyominoes). For our computations, we have translated the problem into an integer linear programming instance.  Finally, using this computational implementation and the game engine Godot, we have developed a video game of minimum domination of queens and rooks on randomly generated polyominoes.

\subjclass{MSC 2020:}{
03D15, 
05B40, 
05B50, 
00A08, 
68Q17, 
68R05, 
68R07 
}
\end{abstract}

\section{Introduction}\label{sec:introduction}
One of the most ancient and famous enumeration problems involving a chessboard and chess pieces is the 8-queens problem that was first stated by Max Bezzel in 1848---see~\cite{bell2009survey,campbell_gauss_1977,LR_accromath_22} for accounts of its history. This problem asks to find the number of different ways of placing 8 queens on a chessboard so that none of them can attack the others (queens can attack vertically, horizontally, and diagonally). Note that $8$ is the maximum number of non-attacking queens that can be placed on an $8\times 8$ chessboard. Many problems have stemmed from this puzzle, some of them are: the $n$--queens problem~\cite{ahrens1918mathematische}---that is, the 8-queens problem generalized to $n\times n$ chessboards for $n \geq 1$; the completion problem~\cite{gent2017complexity,glock_n-queens_2022,nauck1850}, which asks if it is possible to add queens to a given set of non-attacking queens on a $n\times n$ chessboard to complete a non-attacking set with $n$ queens; and the minimum domination problem~\cite[Appendix]{dejaenisch1863}, which consists in finding the minimum number of queens necessary to guard or dominate a chessboard---the known values of this sequence can be found in \cite[\href{https://oeis.org/A075324}{A075324}]{oeis}.

  In this paper, we study domination problems on polyominoes and polycubes\footnote{In 1954, Solomon W. Golomb defined a polyomino as a \emph{finite rook-connected subset of squares of the infinite checkerboard}~\cite{golomb1954checker}. A $d$--polycube is its extension to dimension $d$.} by rooks and queens. This problem is also known as the art gallery problem on polyominoes~\cite{alpert2021art,Biedl12,Iwamoto14,ORourke87}.
 In recent work, the NP-hardness of the minimum domination by rooks and queens was proven for $d$--polycubes for $d\geq 2$~\cite{alpert2021art}---see Table \ref{table:mindom-queens}. In the same paper, the authors studied the maximum non-attacking rooks set problem on polyominoes, and they proved that it is in P.

 We note that the queen problem complexity has been considered on walled chessboards~\cite{barnaby2007} that is equivalent to the problem on path-connected polyominoes---see Theorem~\ref{thm:pathpolyo}.

Puzzles on chessboards have served as testing grounds for computational methods in the last few decades, starting, of course, with the $8$-queens puzzle programming challenge via backtracking popularised by D\ij{}kstra~\cite{Dijkstra71}. Specifically domination problems, like those we are interested in, have attracted a lot of attention since the foundational paper by Cockayne and Hedetniemi~\cite{cockayne1977towards}. Using chessboard problems can help get further results; we refer to the recent review~\cite{HH21} for the domination side and to the recent advance~\cite{BK21} on the $n$-queens problem for specifically this problem. It is also perfectly suited for computational exploration, and it continues to be used a benchmark~\cite{bird_2017,Fischetti18,Kunt23}. In this work, in addition to studying the computational complexity of these domination problems, we also contribute a  general ILP formulation that, combined with state-of-the-art solvers, enabled us to calculate previously unknown minimum numbers of non-attacking queens guarding the $n\times n$ chessboards~\cite[\href{https://oeis.org/A075324}{A075324}]{oeis}  up to $n=31$, from $n=25$~\cite{bird_2017}; see also Table~\ref{tab:known-values} for other new values.

 Our first main results prove the NP-completeness of the minimum domination of attacking and non-attacking rooks and queens on polycubes, extending~\cite[Thms~3,4]{alpert2021art}. To give some perspective, let us note that the minimum domination problem of rooks on a square chessboard is trivially polynomial: a rook is needed in each row and each column, so filling the diagonal solves it for any chessboard. The minimum domination problem for queens on the chessboard has been studied for the last 150 years, yet we still do not know whether there exists a polynomial-time algorithm to find the minimum number of queens needed to dominate a chessboard.

We also study the problem of finding maximum independent sets of queens or rooks on polyominoes; that is, the maximum number of non-attacking rooks or queens that can be placed on a polyomino. In one of our main results, Theorem~\ref{thm:nphard-rooks} (which
answers Question~3 in~\cite{alpert2021art}), we prove that the maximum independent rook domination problem on $d$-polycubes is NP-complete for $d\geq 3$. We also answer the same question for queens in Theorem~\ref{thm:nphard-queens}, proving that the problem is NP-complete for $d\geq 3$.

To put our results in context, in tables~\ref{table:maxdom},~\ref{table:mindom-rooks}, and~\ref{table:mindom-queens} we collect what is known and what is conjectured for the problems of minimum (independent and non-independent) and maximum (independent) domination for rooks and queens. We hope that the information in these tables will help researchers in the field to avoid imprecision on statements about the complexity of these problems and that it will also inspire further research. In what follows, we give precise definitions and statements of our main results.

\begin{table}
    \centering
\begin{tabular} {  c|p{5cm}|p{5cm}}
 \hline
 \hline
 \multicolumn{3}{|c|}{Max Independent Domination Problems} \\
 \hline
Boards & Rooks (non-attacking) & Queens (non-attacking)\\
\hline
\hline
 square completion & P (trivial)  & NP-complete~\cite{gent2017complexity} \\
 \hline
 all polyominoes  &   P~\cite[Thm~12]{alpert2021art}, \newline completion P~\cite[Thm~13]{alpert2021art}  & NP-complete? (Conjecture~\ref{conj:nphard-queens-2d}) \\
 \hline
 $d$--polycubes $d\geq 3$ & NP-complete (Thm~\ref{thm:nphard-rooks})   & NP-complete (Thm~\ref{thm:nphard-queens})\\
 \hline
 \hline
\end{tabular}
\caption{Maximum independent domination problems.}\label{table:maxdom}
\end{table}

\begin{table}
\centering

    \centering
\begin{tabular} {  c|p{5cm}|p{5cm}}
 \hline
 \hline
 \multicolumn{3}{|c|}{Min Rook Domination Problems} \\
 \hline
Boards & attacking & non-attacking (independent) \\
\hline
\hline
 square completion & P (trivial) & P (trivial) \\
 \hline
 all polyominoes & NP-hard \cite[Thm~3]{alpert2021art}, \newline NP-complete (Thm~\ref{thm:minrookNPcomplete}) & NP-hard \cite[Lem.~14]{alpert2021art}, \newline
 NP-complete (Thm~\ref{thm:minrookNPcomplete})  \\
 \hline
 $d$--polycubes $d\geq 3$ & NP-hard \cite[Thm~3]{alpert2021art}, \newline NP-complete (Thm~\ref{thm:minrookNPcomplete})& 
 NP-complete (Thm~\ref{thm:minrookNPcomplete}) \\
 \hline
 \hline

\end{tabular}
\caption{Minimum domination problems for rooks.}\label{table:mindom-rooks}
\end{table}

\begin{table}
\centering

     \centering
\begin{tabular} {  c|p{5cm}|p{5cm}}
 \hline
 \hline
 \multicolumn{3}{|c|}{Min Queen Domination Problem} \\
 \hline
Boards & attacking & non-attacking (independent)\\
\hline
\hline
 square completion & NP-complete? (Conjecture~\ref{conj:min-dom-queens}) & NP-complete (Corollary~\ref{coro:min-ind-dom-npcomp-queens}) \\
 \hline
 all polyominoes & NP-hard \cite[Thm~4]{alpert2021art}, \newline NP-complete (Thm~\ref{thm:minqueenNPcomplete}) & NP-complete (Thm~\ref{thm:minqueenNPcomplete}) \\
 \hline
 $d$--polycubes $d\geq 3$ & NP-hard \cite[Thm~4]{alpert2021art}, \newline NP-complete (Thm~\ref{thm:minqueenNPcomplete})& NP-complete (Thm~\ref{thm:minqueenNPcomplete})  \\
 \hline
 \hline

\end{tabular}
\caption{Minimum domination problems for queens.}\label{table:mindom-queens}
\end{table}

\subsection{Main Results}
We now review the main results of this paper and the relevant definitions.

\begin{definition}
Let $d\geq 2$. A $\mathbf{d}$\textbf{--polycube} is a finite union of unit cubes of the regular cubic tessellation of $\mathbb{R}^d$ with an interior that is connected (notice that polycubes could have holes or cavities). A $2$--polycube is also known as a \textbf{polyomino}.
\end{definition}

We now give a precise definition of the attacking powers of rooks and queens on polycubes. For this purpose, we can imagine the $d$--cubes of the $d$--polycube centered at the points of $\mathbb{Z}^d$.

\begin{definition}[Rook attacking powers]
Suppose a rook is at $(0, \ldots, 0)$ in a $d$--polycube $P$. It guards this point. In addition, for each point that has all its coordinates $0$, except for one coordinate $\pm 1$, we say that the $d$--dimensional rook guards or attacks tiles which have coordinates given by all natural-number multiples of this point such that all the smaller natural-number multiples are tiles of $P$.  \end{definition}

\begin{definition}[Queen attacking powers~\cite{alpert2021art}]
A $d$--dimensional queen placed in a $d$--polycube $P$ at the origin in the cubic tessellation of $\mathbb{R}^d$ can attack all tiles of $P$ with coordinate points equal to 0 or $\pm 1$ and all natural-number multiples of such points, as long as all the smaller natural-number multiples of the points are tiles of $P$.
\end{definition}

Let us note that there are $d$ directions for rooks and $(3^d-1)/2$ for queens, as can be seen by placing the piece at the center of a $d$-hypercube. Figure~\ref{fig:mov-pieces} illustrates these definitions for $d=3$.

Another important note is that the rook's or queen's line of attack ends when it crosses outside the polycube.

\begin{figure}
\centering
 \begin{tikzpicture}[x=(220:0.5cm), y=(-40:0.5cm), z=(90:0.353cm)]
\glasselemcube[sangria]{-1}{-1}{-1}{.1}
\glasselemcube[sangria]{-1}{-1}{0}{.1}
\glasselemcube[sangria]{-1}{-1}{1}{.1}
\glasselemcube[sangria]{-1}{0}{-1}{.1}
\glasselemcube[forestgreen]{-1}{0}{0}{.8}
\glasselemcube[sangria]{-1}{0}{1}{.1}
\glasselemcube[sangria]{-1}{1}{-1}{.1}
\glasselemcube[sangria]{-1}{1}{0}{.1}
\glasselemcube[sangria]{-1}{1}{1}{.1}
\glasselemcube[sangria]{0}{-1}{-1}{.1}
\glasselemcube[forestgreen]{0}{-1}{0}{.8}
\glasselemcube[sangria]{0}{-1}{1}{.1}
\glasselemcube[forestgreen]{0}{0}{-1}{.8}
\glasselemcube[black]{0}{0}{0}{.8}
\cubelabel{0}{0}{0}{\rook}
\glasselemcube[forestgreen]{0}{0}{1}{.8}
\glasselemcube[sangria]{0}{1}{-1}{.1}
\glasselemcube[forestgreen]{0}{1}{0}{.8}
\glasselemcube[sangria]{0}{1}{1}{.1}
\glasselemcube[sangria]{1}{-1}{-1}{.1}
\glasselemcube[sangria]{1}{-1}{0}{.1}
\glasselemcube[sangria]{1}{-1}{1}{.1}
\glasselemcube[sangria]{1}{0}{-1}{.1}
\glasselemcube[forestgreen]{1}{0}{0}{.8}
\glasselemcube[sangria]{1}{0}{1}{.1}
\glasselemcube[sangria]{1}{1}{-1}{.1}
\glasselemcube[sangria]{1}{1}{0}{.1}
\glasselemcube[sangria]{1}{1}{1}{.1}
\end{tikzpicture}
\qquad
\begin{tikzpicture}[x=(220:0.5cm), y=(-40:0.5cm), z=(90:0.353cm)]
\glasselemcube[forestgreen]{-1}{-1}{-1}{.7}
\glasselemcube[forestgreen]{-1}{-1}{0}{.7}
\glasselemcube[forestgreen]{-1}{-1}{1}{.7}
\glasselemcube[forestgreen]{-1}{0}{-1}{.7}
\glasselemcube[forestgreen]{-1}{0}{0}{.8}
\glasselemcube[forestgreen]{-1}{0}{1}{.7}
\glasselemcube[forestgreen]{-1}{1}{-1}{.7}
\glasselemcube[forestgreen]{-1}{1}{0}{.7}
\glasselemcube[forestgreen]{-1}{1}{1}{.7}
\glasselemcube[forestgreen]{0}{-1}{-1}{.7}
\glasselemcube[forestgreen]{0}{-1}{0}{.8}
\glasselemcube[forestgreen]{0}{-1}{1}{.7}
\glasselemcube[forestgreen]{0}{0}{-1}{.8}
\glasselemcube[black]{0}{0}{0}{.8}
\cubelabel{0}{0}{0}{\queen}
\glasselemcube[forestgreen]{0}{0}{1}{.8}
\glasselemcube[forestgreen]{0}{1}{-1}{.7}
\glasselemcube[forestgreen]{0}{1}{0}{.8}
\glasselemcube[forestgreen]{0}{1}{1}{.7}
\glasselemcube[forestgreen]{1}{-1}{-1}{.7}
\glasselemcube[forestgreen]{1}{-1}{0}{.7}
\glasselemcube[forestgreen]{1}{-1}{1}{.7}
\glasselemcube[forestgreen]{1}{0}{-1}{.7}
\glasselemcube[forestgreen]{1}{0}{0}{.8}
\glasselemcube[forestgreen]{1}{0}{1}{.7}
\glasselemcube[forestgreen]{1}{1}{-1}{.7}
\glasselemcube[forestgreen]{1}{1}{0}{.7}
\glasselemcube[forestgreen]{1}{1}{1}{.7}
\glasselemcube[forestgreen]{-2}{-2}{2}{.9}
\glasselemcube[sangria]{-2}{-1}{2}{.3}
\glasselemcube[forestgreen]{-2}{0}{2}{.9}
\glasselemcube[sangria]{-2}{1}{2}{.3}
\glasselemcube[forestgreen]{-2}{2}{2}{.9}
\glasselemcube[sangria]{-1}{-2}{2}{.3}
\glasselemcube[sangria]{-1}{-1}{2}{.3}
\glasselemcube[sangria]{-1}{0}{2}{.3}
\glasselemcube[sangria]{-1}{1}{2}{.3}
\glasselemcube[sangria]{-1}{2}{2}{.3}
\glasselemcube[forestgreen]{0}{-2}{2}{.9}
\glasselemcube[sangria]{0}{-1}{2}{.3}
\glasselemcube[forestgreen]{0}{0}{2}{.9}
\glasselemcube[sangria]{0}{1}{2}{.3}
\glasselemcube[forestgreen]{0}{2}{2}{.9}
\glasselemcube[sangria]{1}{-2}{2}{.3}
\glasselemcube[sangria]{1}{-1}{2}{.3}
\glasselemcube[sangria]{1}{0}{2}{.3}
\glasselemcube[sangria]{1}{1}{2}{.3}
\glasselemcube[sangria]{1}{2}{2}{.3}
\glasselemcube[forestgreen]{2}{-2}{2}{.9}
\glasselemcube[sangria]{2}{-1}{2}{.3}
\glasselemcube[forestgreen]{2}{0}{2}{.9}
\glasselemcube[sangria]{2}{1}{2}{.3}
\glasselemcube[forestgreen]{2}{2}{2}{.9}
\end{tikzpicture}
\qquad
\begin{tikzpicture}[x=(220:0.5cm), y=(-40:0.5cm), z=(90:0.353cm)]
\draw[->] (0,0,0) -- (2,0,0) node[left]{x};
\draw[->] (0,0,0) -- (0,2,0) node[right]{y};
\draw[->] (0,0,0) -- (0,0,2) node[above]{z};
\end{tikzpicture}
\caption{Possible movement (in green) of a rook and a queen centered in a $3\times 3\times 3$ cube. For the queen, a $5\times 5$ level is put at the top of the cube. 3D models corresponding to these structures can be found at \url{https://skfb.ly/oz8tJ} and \url{https://skfb.ly/oz8tn}, respectively.}
\label{fig:mov-pieces}

\end{figure}

Our first two main results study the computational complexity of a class of maximum independent domination problems on polycubes.

\begin{theorem}\label{thm:nphard-queens}
Solving the maximum non-attacking queen domination problem on $d$--polycubes is NP-complete for $d \geq 3$.
\end{theorem}

\begin{theorem}\label{thm:nphard-rooks}
Solving the maximum non-attacking rook domination problem on $d$--polycubes is NP-complete for $d \geq 3$.
\end{theorem}

It was proven that, for $d=2$,  there is a polynomial-time algorithm  that solves the maximum domination problem for non-attacking rooks~\cite[Thm~12]{alpert2021art}. The passage to three dimensions is the crucial point at which the complexity of this problem changes.

We also complete the proof of
NP-completeness of the minimum domination problem for non-attacking and attacking queens and rooks, extending the proof of NP-hardness of~\cite[Thms~3,4]{alpert2021art} to the cases of non-attacking pieces, and then showing they are all NP-complete.
\begin{theorem}\label{thm:minrookNPcomplete}
Solving the minimum domination problem for attacking and for non-attacking rooks on $d$--polycubes for $d\geq 2$ is NP-complete.
\end{theorem}

\begin{theorem}\label{thm:minqueenNPcomplete}
Solving the minimum domination problem for attacking and for non-attacking queens on $d$--polycubes for $d\geq 2$ is NP-complete.
\end{theorem}

We now review the structure of the paper. In Section~\ref{sec:prelims}, we present the relevant definitions and prove Theorems~\ref{thm:minrookNPcomplete} and~\ref{thm:minqueenNPcomplete}. In Section~\ref{sec:complexity},	 we prove Theorems~\ref{thm:nphard-queens} and~\ref{thm:nphard-rooks}. In Section~\ref{sec:openquestions}, we study domination problems specifically on polyominoes, in particular on convex polyominoes, and state a series of open questions and conjectures. In Section~\ref{sec:modelcompu}, we present a computational model and use it to compute new values of domination numbers for classical problems. The software developed and implemented in the course of this research is publicly available on GitHub~\cite{QandRsoftware}. Also, in Section~\ref{sec:modelcompu}, we give a brief description of a video game on chess domination on polyominoes; the game can be played at this link: \url{https://www.erikaroldan.net/queensrooksdomination}.

\section{Preliminaries}\label{sec:prelims}

We now define the two problems we will consider.
\begin{definition}[Independent rook domination] We say that an instance of the non-attacking rook set problem for $d$--polycubes is a pair $(P,m)_{d}^R$, where $P$ is a $d$--polycube and $m$ is a positive integer. The problem asks whether there exists a non-attacking configuration of $m$ rooks placed in $P$ that dominates $P$.
\end{definition}

\begin{definition}[Independent queen domination] We say that an instance of the non-attacking queen set problem for $d$--polycubes is a pair $(P,m)_{d}^Q$, where $P$ is a $d$--polycube and $m$ is a positive integer. The problem asks whether there exists a non-attacking configuration of $m$ queens placed in $P$ that dominates $P$.
\end{definition}

We will study the complexity class of the maximum and minimum $m$ possible for both instances $(P,m)_d^Q$ and $(P,m)_d^R$. To prove NP-complexity, we follow the usual process. To help readers follow our proofs, we briefly review the steps here.  We first show the problem is in the class NP by showing that verifying a solution is done in polynomial time. This is shown in Lemma~\ref{lem:verifP} below.

 We then exhibit a polynomial reduction from a known NP-complete problem. This will be the purpose of Section~\ref{sec:complexity}. In our case, we will begin from the following restriction of \textsc{planar sat}.

 \begin{definition}[\psat~\cite{Cerioli08}]
 Given a set of Boolean variables 
 $x_i$ that satisfy a set of clauses of the form $x_{i_1}\lor x_{i_2}$ or $x_{i_1}\lor x_{i_2} \lor x_{i_3}$, with each $x_{i_k}$ appearing being the literal $x_{i_k}$ or $\overline{x}_{i_k}$, we construct a bipartite graph with the two sets of vertices given by the variables and the clauses, and with the set of edges given by linking a clause $C$ and a variable $x$ if the clause $C$ contains either literals $x$ or $\overline{x}$. The problem of deciding whether there exists a truth assignment to the variables such that all clauses are satisfied is called $\textsc{3sat}$. If the bipartite graph constructed is planar, this problem is instead called \textsc{planar 3sat}. If, furthermore, each variable $x_i$ appears exactly in three clauses, we call the problem \textsc{planar 3sat with exactly 3 occurrences per variable}, and it is denoted as \psat.
 \end{definition}

  An example of such instance is given in Figure~\ref{fig:ex_psat}.

\begin{figure}[h]
\centering
\begin{tikzcd}
 c_2=\overline{x}_1\lor \overline{x}_2 \arrow[r, dash] \arrow[dd,dash] & x_1 \arrow[d,dash] \arrow[r,dash] & c_4= x_1 \lor \overline{x}_3 \arrow[dd, dash]\\
 & c_1=x_1 \lor x_2 \lor x_3 \arrow[ld,dash] \arrow[dr,dash]& \\
 x_2 \arrow[r,dash] & c_3=\overline x_2\lor \overline x_3 & \arrow[l,dash] x_3
\end{tikzcd}
\caption{An instance of $\psat$ with three variables, $x_1$, $x_2$ and $x_3$, and four clauses, $c_1=x_1\lor x_2\lor x_3$, $c_2=\overline x_1\lor \overline x_2$, $c_3=\overline x_2\lor \overline x_3$ and $c_4=x_1\lor \overline x_3$. There are three solutions: $(x_1,x_2,x_3) = (1,0,1),\ (1,0,0),\ (0,1,0)$}\label{fig:ex_psat} 
\end{figure}

  \begin{remark}
   The problem \psat{} was proven to be NP-complete~\cite{Cerioli08}, with the further restriction that each variable appears once positively and twice negatively. This was the problem used for the proofs in~\cite{alpert2021art}. We do not require this additional restriction. Note, however, that with the further restriction that all clauses contain exactly 3 literals, the problem becomes solvable in polynomial time~\cite{BSK03}.
  \end{remark}

 \begin{proposition}[\cite{Cerioli08}]\label{prop:psatNP}
 The problem \psat{} is NP-complete.
 \end{proposition}
The reduction will proceed by introducing \textbf{gadgets} to encode variables, connections, and clauses as maximum independent rooks or queens set problems on polycubes. The main difficulty is how to show that the size of the polycubes created is polynomially bounded, and that the algorithm to do so runs in polynomial time.

Let us begin our proof of complexity by showing that the verification of a candidate for domination and independence can be done in polynomial time.
\begin{lemma}\label{lem:verifP}
Verifying that a placement of rooks or a placement of queens dominate a $d$--polycube and verifying that a placement of queens or rooks that are not attacking each other can be done in polynomial time.
\end{lemma}
\begin{proof}
We recall that we can construct a graph from a $d$-polycube and a given choice of chess pieces (rook or queen). The vertices are the $d$-cubes and there is an edge between two vertices if the chosen piece can reach one $d$-cube from the other. Asking whether a set of pieces dominates the $d$-polycube is then equivalent to verifying if a given set of vertices (the $d$-cubes occupied by the pieces) dominates the graph. Similarly, asking whether pieces are not attacking each other is equivalent to asking if a given set of vertices forms a coclique, that is, a set of vertices with no edge between each other. Both of these problems are verifiable polynomially---see~\cite{PS98}. This proves that the solution is verifiable in polynomial time.
\end{proof}

To end this section, we give the proofs of Theorem~\ref{thm:minrookNPcomplete} and Theorem~\ref{thm:minqueenNPcomplete}. They are easy extensions of the proof of NP-hardness of minimum domination of (potentially) attacking rooks or queens~\cite{alpert2021art}.

\begin{proof}[Proof of Theorem~\ref{thm:minrookNPcomplete}]
By Lemma~\ref{lem:verifP}, we know that verifying the validity of a given set of rooks for the minimum domination problem for attacking and non-attacking rooks can be checked in polynomial time, and so the problems are in the class NP. The proof of~\cite[Thm~3]{alpert2021art} for attacking rooks gives that it is NP-hard, so it is in fact NP-complete.

Finally, we note that the gadgets used in the proof of~\cite[Thm~3]{alpert2021art} also hold true for non-attacking rooks; thus, the domination problem for non-attacking rooks is also NP-complete.
\end{proof}

\begin{proof}[Proof of Theorem~\ref{thm:minqueenNPcomplete}]
That a set of queens guards a polycube and that they do not attack each other can be checked in polynomial time by Lemma~\ref{lem:verifP}, so the minimum domination problem for attacking and non-attacking queens is in the class NP. The proof of~\cite[Thm~4]{alpert2021art} also holds for non-attacking queens since the gadgets only use non-attacking setups, and so it proves that both problems are NP-hard. Hence, they are NP-complete.
\end{proof}

\begin{remark}
    In many cases, for polycubes, the minimum dominating sets of attacking or non-attacking rooks or queens will differ; the latter being generally bigger. For queens, this already happens for polyominoes. For rooks, the solutions will be of the same size on polyominoes by~\cite[Lemma~14]{alpert2021art}, but they are different on higher-dimension polycubes. Figure~\ref{fig:diff_non-att_and_att} shows two small examples where the attacking case has a smaller solution. The proof of NP-hardness of~\cite{alpert2021art}, in fact, considered the smaller subset of polyominoes for which both attacking and non-attacking problems have the same solution, thus proving NP-hardness for general polycubes since they include this subset for which the problem is NP-hard.
\end{remark}

\begin{figure}[h]
\begin{align*}
   \begin{tikzpicture}
        \foreach \x/\y in {1/0,3/0,1/1,3/1,2/2,1/3,3/3,1/4,3/4, 2/4, 2/0 } {
            \path [draw=brown!70, fill=brown!70] (+3.5+\x-0.45, 1.5+\y-1.45)
            -- ++(0,.9)
            -- ++(.9,0)
            -- ++(0,-.9)
            --cycle;
        }
        \foreach \x/\y in {2/1,2/3 } {
            \path [draw=teal!70, fill=teal!70] (+3.5+\x-0.45, 1.5+\y-1.45)
            -- ++(0,.9)
            -- ++(.9,0)
            -- ++(0,-.9)
            --cycle;
             \node[anchor=west] at (+3.5+\x-0.5, \y+.5) {\queen};
        }
        \end{tikzpicture}
        &&
         \begin{tikzpicture}[x=(220:0.5cm), y=(-40:0.5cm), z=(90:0.353cm)]
\glasselemcube[brown]{0}{-1}{0}{.75}
\glasselemcube[brown]{-1}{0}{1}{.75}
\glasselemcube[teal]{0}{0}{0}{.95}
\cubelabel{0}{0}{0}{\rook}
\glasselemcube[brown]{1}{0}{0}{.75}
\glasselemcube[teal]{0}{0}{1}{.95}
\cubelabel{0}{0}{1}{\rook}
\glasselemcube[brown]{0}{1}{1}{.75}
\end{tikzpicture}
\end{align*}
    \caption{Left: an instance of a polyomino with attacking dominating set of 2 queens and non-attacking dominating set of 3 queens; right: an instance of a polycube with attacking dominating set of 2 rooks and non-attacking of 3 rooks.}
    \label{fig:diff_non-att_and_att}
\end{figure}

\section{Complexity of the domination problems}\label{sec:complexity}
This section is dedicated to proving Theorems~\ref{thm:nphard-queens} and~\ref{thm:nphard-rooks}. We will begin with the rooks and proceed to the queens.

\subsection{NP-Completion of maximum independent domination of rooks on polycubes}

We first introduce the variable gadget for non-attacking rooks in Figure~\ref{fig:variable-gadget-rooks}. There are two ways to place the maximum number of rooks on the gadget.
\begin{figure}[h]
    \begin{center}
    \begin{tikzpicture}[x=(220:0.5cm), y=(-40:0.5cm), z=(90:0.353cm)]
    \foreach \x/\y in {-2/0}
    {
        \elemcube{\x}{\y}{0}
    }
    \elemcube[black]{-1}{-1}{0}
    \foreach \x/\y in {-1/0}
    {
        \elemcube{\x}{\y}{0}
    }
    \elemcube[sangria]{-1}{1}{0}
    \foreach \x/\y in {0/-2, 0/-1, 0/1,0/2}
    {
    \elemcube{\x}{\y}{0}
    }
    \elemcube[sangria]{1}{-1}{0}
    \foreach \x/\y in {1/0}
    {
        \elemcube{\x}{\y}{0}
    }
    \elemcube[black]{1}{1}{0}
    \foreach \x/\y in {2/0}
    {
    \elemcube{\x}{\y}{0}
    }

\foreach \x/\y in {-1/-1,1/1}
{
    \cubelabel{\x}{\y}{0}{\textbf{T}}
}
\foreach \x/\y in {-2/0,0/-2,2/0,0/2}
{
    \cubelabel{\x}{\y}{0}{\rook}
}
    \foreach \x/\y in {1/-1,-1/1}
{
    \cubelabel{\x}{\y}{0}{\textbf{F}}
}
\end{tikzpicture}
\qquad \qquad \qquad
\begin{tikzpicture}
    \draw[->] (0,0) -- (1,0) node[right] {E};
    \draw[->] (0,0) -- (0,1) node[above] {N};
    \draw[->] (0,0) -- (0,-1) node[below] {S};
    \draw[->] (0,0) -- (-1,0) node[left] {W};
\end{tikzpicture}
    \end{center}
    \caption{Variable gadget with rooks; when \true, 2 additional rooks go on the dark cubes (T), and when \false, 2 go on the light red cubes (F). A 3D model corresponding to this structure can be found at \url{https://skfb.ly/oz8tZ}. The general orientation is given next to it.}
    \label{fig:variable-gadget-rooks}
\end{figure}

\begin{lemma}\label{lem:variable-gadget-rooks}
The maximum number of dominating non-attacking rooks for the polycube of Figure~\ref{fig:variable-gadget-rooks} is 6, and there are only 2 ways to achieve this maximum number.
\end{lemma}
\begin{proof}
The middle holed section can be guarded either by 4 non-attacking rooks in a cross pattern, which would dominate the full polyomino, or by 2 rooks in opposite corners, which then requires 4 rooks on the remaining unguarded cubes. Thus, at most 6 non-attacking rooks can guard the polyomino.

Four rooks can be placed on the protruding corners. Then two additional ones can either be placed on the two dark T cubes, for which we say the variable has truth value \true, or they can be placed on the two red F cubes, for which we say the variable has truth value \false.
\end{proof}

\begin{figure}[h!]
\begin{subfigure}{.45\textwidth}
\centering
\begin{tikzpicture}[x=(220:0.5cm), y=(-40:0.5cm), z=(90:0.353cm)]
\foreach \x/\y/\z in {-2/-2/2, -1/-2/1, -1/-1/1, -1/-2/2, -1/-2/3, 0/-2/1, 0/-2/3, 0/-2/4, 1/-2/1}
    {
        \elemcube{\x}{\y}{\z}
    }
\elemcube[sangria]{1}{-2}{1}
\foreach \x/\y/\z in {1/-2/2, 1/-2/3, 2/-2/2}
    {
    \elemcube{\x}{\y}{\z}
    }
      \cubelabel{-2}{-2}{2}{\rook}
    \elemcube[sangria]{-1}{-1}{1}
    \elemcube[black]{1}{-1}{1}
    \colourcubeXface[sangria]{-1}{-2}{3}
     \colourcubeZface[sangria]{-1}{-2}{3}
    \elemcube[black]{1}{-2}{3}
    \cubelabel{-1}{-2}{3}{\textbf{F}}
    \cubelabel{1}{-2}{3}{\textbf{T}}
    \cubelabel{-1}{-1}{1}{\textbf{F}}
    \cubelabel{1}{-1}{1}{\textbf{T}}
    \elemcube{0}{-2}{4}
    \foreach \x/\y/\z in {2/-2/2,0/-2/4}
    {
    \cubelabel{\x}{\y}{\z}{\rook}
    }
    \end{tikzpicture}
    \quad \quad
    \begin{tikzpicture}[x=(220:0.5cm), y=(-40:0.5cm), z=(90:0.353cm)]
    \elemcube[sangria]{0}{-1}{0}
    \elemcube{0}{0}{0}
    \elemcube[black]{0}{1}{0}
    \elemcube[black]{1}{-1}{0}
    \elemcube[sangria]{1}{1}{0}
\foreach \x/\y/\z in {0/-2/1,0/-1/1,0/1/1,0/2/1}
    {
        \elemcube{\x}{\y}{\z}
    }
      \foreach \x/\y/\z in {0/-2/1,0/2/1}
     {
     \cubelabel{\x}{\y}{\z}{\rook}
     }
    \elemcube[black]{0}{-1}{2}
    \elemcube{0}{0}{2}
    \elemcube[sangria]{0}{1}{2}
     \cubelabel{1}{1}{0}{\textbf{F}}
     \cubelabel{1}{-1}{0}{\textbf{T}}
     \cubelabel{0}{1}{2}{\textbf{F}}
     \cubelabel{0}{-1}{2}{\textbf{T}}
\elemcube{0}{0}{3}
     \foreach \x/\y/\z in {0/0/3}
     {
     \cubelabel{\x}{\y}{\z}{\rook}
     }
    \end{tikzpicture}
    \caption{The connection gadget (two views). \\ \url{https://skfb.ly/ozJZn}}
  \end{subfigure}
  \begin{subfigure}{.45\textwidth}
  \centering
    \begin{tikzpicture}[x=(220:0.5cm), y=(-40:0.5cm), z=(90:0.353cm)]
\foreach \x/\y in {-2/0}
    {
        \elemcube{\x}{\y}{0}
    }
    \elemcube[black]{-1}{-1}{0}
    \foreach \x/\y in {-1/0}
    {
        \elemcube{\x}{\y}{0}
    }
    \elemcube[sangria]{-1}{1}{0}
    \foreach \x/\y in {0/-2, 0/-1, 0/1 ,0/2}
    {
    \elemcube{\x}{\y}{0}
    }
    \elemcube[sangria]{1}{-1}{0}
    \foreach \x/\y in {1/0}
    {
        \elemcube{\x}{\y}{0}
    }
    \elemcube[black]{1}{1}{0}
    \foreach \x/\y in {2/0}
    {
    \elemcube{\x}{\y}{0}
    }
\foreach \x/\y in {-1/-1,1/1}
{
    \cubelabel{\x}{\y}{0}{\textbf{T}}
}
\foreach \x/\y in {-2/0,0/-2,2/0,0/2}
{
    \cubelabel{\x}{\y}{0}{\rook}
}
    \foreach \x/\y in {1/-1,-1/1}
{
    \cubelabel{\x}{\y}{0}{\textbf{F}}
}
\foreach \x/\y/\z in {-2/-2/2, -1/-2/1, -1/-1/1, -1/-2/2, -1/-2/3, 0/-2/1, 0/-2/3, 0/-2/4}
    {
        \elemcube{\x}{\y}{\z}
    }
\elemcube[sangria]{1}{-2}{1}
\foreach \x/\y/\z in {1/-2/2, 1/-2/3, 2/-2/2}
    {
    \elemcube{\x}{\y}{\z}
    }
      \cubelabel{-2}{-2}{2}{\rook}
    \elemcube[sangria]{-1}{-1}{1}
    \elemcube[black]{1}{-1}{1}
    \colourcubeXface[sangria]{-1}{-2}{3}
     \colourcubeZface[sangria]{-1}{-2}{3}
    \elemcube[black]{1}{-2}{3}
    \cubelabel{-1}{-2}{3}{\textbf{F}}
    \cubelabel{1}{-2}{3}{\textbf{T}}
        \cubelabel{-1}{-1}{1}{\textbf{F}}
    \cubelabel{1}{-1}{1}{\textbf{T}}
    \elemcube{0}{-2}{4}
    \foreach \x/\y/\z in {2/-2/2,0/-2/4}
    {
    \cubelabel{\x}{\y}{\z}{\rook}
    }
\end{tikzpicture}
\caption{Connection gadget on a variable gadget. \\ \url{https://skfb.ly/ozV7S}}
\end{subfigure}
\caption{The connection gadget for rooks; values are transmitted along the T and F cubes for \true{} and \false{}, respectively. There are some cubes not shown on the image. (A): the left has a T cube not shown; right has an F cube not shown; (B): two T cubes not shown behind and under the F cube of the connection gadget on the variable gadget}\label{fig:connection-gadget-rooks}
\end{figure}

To transmit the signal, we will need to use the third dimension. Note, however, that the gadgets will never need to exceed 9 cubes in height.  Let us introduce the connection gadget in Figure~\ref{fig:connection-gadget-rooks}. It has a maximum covering number by rooks of 6. It is placed on the variable gadget on one of the NW or SW corners of the variable gadget, with the protruding T and F tiles on F and T tiles, respectively. According to the variable gadget truth value, the maximum independent covering of rooks on the connection gadget will add rooks on the T or F tiles.

 \begin{lemma}\label{lem:connection-gadget-rooks}
 The signal of the variable gadget is propagated by the connection gadget of Figure~\ref{fig:connection-gadget-rooks}. Furthermore, the connection gadget can be used to turn corners.
 \end{lemma}
 \begin{proof}
 As we connect to the variable gadget by placing a red F tile on a dark T tile and a dark T tile on a red F tile, the truth value of the variable gadget will influence where we can place the maximum 6 rooks on the connection gadget. First, 3 rooks are placed on the corners. Then, if the variable has the value \true,  the remaining rooks are placed on T tiles. Specifically, one is placed on the T tile above the variable gadget, one is placed on the cube just behind the F cube above the variable gadget, and one is placed on the top T cube of the connection gadget. This process can be repeated on the connection gadget to place a variable tile at the top. One can also remain on the same level by repeating the process a second time. The front view of the resulting process  can be seen
 in Figure~\ref{fig:cutsection-connection-rooks}.
 \begin{figure}[h]
 \centering
 \begin{tikzpicture}[scale=1]
        \foreach \x/\y in {1/1,2/1,3/1,4/1,4/2,5/1,5/2,5/3,5/4,5/5,6/4,6/5,7/5,8/4,8/5,9/5,9/4,9/3,9/2,9/1,10/2,10/1,11/1,12/1,13/1} {
            \path [draw=brown!, fill=brown!30] (.5+\x-0.45, .5+\y-1.45)
            -- ++(0,.9)
            -- ++(.9,0)
            -- ++(0,-.9)
            --cycle;
        }

        \foreach \x/\y in {1/1,3/1,5/1,5/3,5/5,7/5,9/5,9/3,9/1,11/1,13/1} {
            \node at (.5+\x, \y-.5) {\rook};
        }
          \foreach \x/\y in {4/1,5/2,6/4,8/5,9/4,10/2,12/1} {
          \path [draw=gray!, fill=gray!30] (.5+\x-0.45, .5+\y-1.45)
            -- ++(0,.9)
            -- ++(.9,0)
            -- ++(0,-.9)
            --cycle;
            \node[anchor=west] at (.5+\x-0.4, \y-.5) {\textbf{T}};
        }
          \foreach \x/\y in {2/1,4/2,5/4,6/5,8/4,9/2,10/1} {
          \path [draw=sangria!, fill=sangria!30] (.5+\x-0.45, .5+\y-1.45)
            -- ++(0,.9)
            -- ++(.9,0)
            -- ++(0,-.9)
            --cycle;
            \node[anchor=west] at (.5+\x-0.4, \y-.5) {\textbf{F}};
        }
    \end{tikzpicture}
    \caption{Propagating the signal: a 2D front view of the connection; 11 rooks are not shown. }\label{fig:cutsection-connection-rooks}
\end{figure}

Finally, we can go around corners by rotating the connection gadget, always propagating the signal, and duplicating that signal if needed, as shown in Figure~\ref{fig:corner-gadget-rooks}. The height we need for this is 9 cubes.
 \end{proof}

 \begin{figure}[h]
 \centering
 \begin{subfigure}{.45\textwidth}
 \centering
 \begin{tikzpicture}[x=(220:0.5cm), y=(-40:0.5cm), z=(90:0.353cm)
 ]
\foreach \x/\y/\z in {-2/0/0,-1/-1/0, -1/0/0,-1/1/0, 0/-2/0, 0/-1/0, 0/1/0 ,0/2/0,1/-1/0, 1/0/0,1/1/0,2/0/0}
    {
    \elemcube{\x}{\y}{\z}
    }
\elemcube{-1}{-3}{3}
\elemcube{1}{-3}{3}
\foreach \x/\y/\z in {-2/-2/2, -1/-2/1, -1/-1/1, -1/-2/2, -1/-2/3, 0/-2/1, 0/-2/3, 0/-2/4, 1/-2/1, 1/-1/1, 1/-2/2, 1/-2/3, 2/-2/2}
    {
    \elemcube{\x}{\y}{\z}
    }
    \foreach \x/\y/\z in {-2/0/0,-1/-1/0, -1/0/0,-1/1/0}
     {
     \elemcube{\x}{\y-4}{\z+4}
     }
\foreach \x/\y/\z in 
{0/-2/0,0/-1/-1,0/-1/0,0/-1/1,,0/0/-1,0/0/1,0/0/2,0/1/-1,0/1/0,0/1/1,0/2/0,1/-1/-1,1/1/-1}
    {
    \elemcube{\x-2}{\y-4}{\z+6}
    }
\foreach \x/\y/\z in { 0/-2/0, 0/-1/0, 0/1/0 ,0/2/0,1/-1/0, 1/0/0,1/1/0,2/0/0}
    {
    \elemcube{\x}{\y-4}{\z+4}
    }

\end{tikzpicture}
\caption{Turning the corner with connection gadgets.\\ \url{https://skfb.ly/ozUTt}}
 \end{subfigure}
 \begin{subfigure}{.45\textwidth}
 \centering
\begin{tikzpicture}[x=(220:0.5cm), y=(-40:0.5cm), z=(90:0.353cm)
]
\foreach \x/\y/\z in {-2/0/0,-1/-1/0, -1/0/0,-1/1/0, 0/-2/0, 0/-1/0, 0/1/0 ,0/2/0,1/-1/0, 1/0/0,1/1/0,2/0/0}
    {
    \elemcube{\x}{\y}{\z}
    }
\elemcube{-1}{-3}{3}
\elemcube{1}{-3}{3}
\foreach \x/\y/\z in {-2/-2/2, -1/-2/1, -1/-1/1, -1/-2/2, -1/-2/3, 0/-2/1, 0/-2/3, 0/-2/4, 1/-2/1, 1/-1/1, 1/-2/2, 1/-2/3, 2/-2/2}
    {
    \elemcube{\x}{\y}{\z}
    }
    \foreach \x/\y/\z in {-2/0/0,-1/-1/0, -1/0/0,-1/1/0}
     {
     \elemcube{\x}{\y-4}{\z+4}
     }
\foreach \x/\y/\z in 
{0/-2/0,0/-1/-1,0/-1/0,0/-1/1,0/0/-1,0/0/1,0/0/2,0/1/-1,0/1/0,0/1/1,0/2/0,1/-1/-1,1/1/-1}
    {
    \elemcube{\x-2}{\y-4}{\z+6}
    }
\foreach \x/\y/\z in { 0/-2/0, 0/-1/0, 0/1/0 ,0/2/0,1/-1/0, 1/0/0,1/1/0,2/0/0}
    {
    \elemcube{\x}{\y-4}{\z+4}
    }
\foreach\x/\y/\z in {0/-2/0,0/-1/-1,0/-1/0,0/-1/1,0/0/-1,0/0/1,0/0/2,0/1/-1,0/1/0,0/1/1,0/2/0}
{
\elemcube{\x+2}{\y-4}{\z+6}
}
\end{tikzpicture}
\caption{Duplicating the signal with connection gadgets. \\ \url{https://skfb.ly/ozUSO}}
\end{subfigure}
 \caption{Connection gadgets propagating and duplicating the signal.}\label{fig:corner-gadget-rooks}
 \end{figure}

We introduce the clause gadgets in Figure~\ref{fig:clause-2gadget-rooks}. The clause gadget is created as follows. Let $c$ be a clause and suppose that the variable gadgets are aligned on the NE axis. We place a connector %
\begin{tikzpicture}[x=(200:0.5cm), y=(-40:0.5cm), z=(90:0.353cm),baseline= {(current bounding box.center)}, scale = .25]
\foreach \x/\y/\z in {0/-1/0, 0/0/0,0/0/1}
{\elemcube[forestgreen]{\x}{\y}{\z}}
\end{tikzpicture}
on top of the NW red F tile in the variable $x$ if $x$ appears as $x\in c$ and on the NW gray T tile if it is negated. The connectors are joined together by alternating top %
\begin{tikzpicture}[x=(200:0.5cm), y=(-40:0.5cm), z=(90:0.353cm),baseline= {(current bounding box.center)}, scale = .25]
\foreach \x/\y/\z in { 0/0/1,0/0/2}
{\elemcube[forestgreen]{\x}{\y}{\z}}
\end{tikzpicture}
and left %
\begin{tikzpicture}[x=(200:0.5cm), y=(-40:0.5cm), z=(90:0.353cm),baseline= {(current bounding box.center)}, scale = .25]
\foreach \x/\y/\z in { 0/0/1,0/1/1}
{\elemcube[forestgreen]{\x}{\y}{\z}}
\end{tikzpicture}
nodes in a line. We denote $\ell_R(c)$ as the length of the clause gadget associated with $c$
    \begin{equation}\label{eq:f_clause}
        \ell_R(c) := |\{\text{nodes and connectors in $c$}\}|.
    \end{equation}
Figure~\ref{fig:clause-gadget-connection-rooks} presents one example.

\begin{lemma}\label{lem:clause-gadget-rooks}
Let $c$ be a  clause and $K(c)$ be the clause gadget associated with it. A maximum rook placement on $K(c)$ contains $\ell_R(c)$ rooks if $c$ evaluates to \false{} and $\ell_R(c)+1$ if it evaluates to \true.
\end{lemma}


\def\inc{1} 
\begin{figure}[h]
\begin{subfigure}{.45\textwidth}
\centering
\begin{tikzpicture}[rotate = -50,x=(245:0.5cm), y=(-30:0.5cm), z=(90:0.353cm)
]
\foreach \x/\y/\z in {0/-1/1,6/-1/1}
{\elemcube[forestgreen]{\x}{\y}{\z}}
\foreach \x/\y/\z in {0/-1/1,6/-1/1}
    {
        \cubelabel{\x}{\y}{\z}{\rook}
    }
\foreach \x/\y/\z in {0/0/1}
{\elemcube[forestgreen]{\x}{\y}{\z}}
\elemcube[black]{0}{0}{2}
\colourcubeZface[black]{0}{0}{2}
\cubelabel{0}{0}{2}{\textbf{T}}
\foreach \x/\y/\z in {1/0/2,1/0/3,2/0/2,2/1/2,3/0/2,3/0/3,4/0/2,4/1/2,5/0/2,5/0/3,6/0/1}
{\elemcube[forestgreen]{\x}{\y}{\z}}
\elemcube[black]{6}{0}{2}
\colourcubeZface[black]{6}{0}{2}
\cubelabel{6}{0}{2}{\textbf{T}}
\foreach \x/\y/\z in {1/0/3,2/1/2,3/0/3,4/1/2,5/0/3}
{\cubelabel{\x}{\y}{\z}{\rook}}
\end{tikzpicture}
\caption{Clause gadget linked to the clauses with two positive  or two negated literals.\\ \url{https://skfb.ly/ozVqt}}
\end{subfigure}
\begin{subfigure}{.45\textwidth}
\centering
\begin{tikzpicture}[rotate = -50,x=(245:0.5cm), y=(-30:0.5cm), z=(90:0.353cm)
]
\foreach \x/\y/\z in {6/-1/1,10/-1/1}
{\elemcube[forestgreen]{\x}{\y}{\z}}
\foreach \x/\y/\z in {6/-1/1,10/-1/1}
    {
        \cubelabel{\x}{\y}{\z}{\rook}
    }
    \foreach \x/\y/\z in {6/0/1}
{\elemcube[forestgreen]{\x}{\y}{\z}}
\elemcube[black]{6}{0}{2}
\colourcubeZface[black]{6}{0}{2}
\cubelabel{6}{0}{2}{\textbf{T}}
\foreach \x/\y/\z in {7/0/2,7/1/2,8/0/2,8/0/3,9/0/2,9/1/2}
{\elemcube[forestgreen]{\x}{\y}{\z}}
\elemcube{10}{0}{1}
\elemcube[black]{10}{0}{2}
\colourcubeZface[black]{10}{0}{2}
\cubelabel{10}{0}{2}{\textbf{T}}
\foreach \x/\y/\z in {11/0/2,11/1/2,12/0/2,12/0/3}
{\elemcube[forestgreen]{\x}{\y}{\z}}
\foreach \x/\y/\z in {7/1/2,8/0/3,9/1/2,11/1/2,12/0/3}
{\cubelabel{\x}{\y}{\z}{\rook}}
\end{tikzpicture}
\caption{Clause gadget linked to the clauses with one positive literal and one negated literal.\\ \url{https://skfb.ly/ozVqH}}
\end{subfigure}
\caption{The clause gadgets with two literals, all light green cubes are guarded, and one additional rook might be placed on one of the two dark T cubes.}\label{fig:clause-2gadget-rooks}
\end{figure}

\begin{proof}
Construct the clause gadget $K(c)$ from the clause $c$. On each top and left nodes there can be one rook, and there can be at most two rooks on the connectors. However, because the connectors are all placed on a line, only one can have two rooks. There is one node or connector for each tile; thus, $\ell_R(c)$. By construction, as $K$ is placed atop the variable, there can only be an extra rook if the clause evaluates to \true. For example, Figure~\ref{fig:clause-gadget-connection-rooks} presents a clause gadget $K(c)$ of length 13 for the clause $c=\mathbf{x_1}\lor \mathbf{x_2}\lor \mathbf{x_3}$. No extra rook can be placed if all literals are \false{}.
\end{proof}

\begin{figure}[h]
\centering
\begin{tikzpicture}[x=(200:0.5cm), y=(-40:0.5cm), z=(90:0.353cm)]
  \foreach \x/\y in {-2/0,-1/-1, -1/0,-1/1, 0/-2, 0/-1, 0/1 ,0/2,1/-1, 1/0,1/1,2/0}
    {
    \elemcube{\x-\inc}{\y+\inc}{0}
    }
\colourcubeZface[black]{-1-\inc}{-1+\inc}{0}
\colourcubeZface[black]{1-\inc}{1+\inc}{0}
\colourcubeZface[sangria]{-1-\inc}{1+\inc}{0}
\colourcubeZface[sangria]{1-\inc}{-1+\inc}{0}
\foreach \x/\y in {1/1,-1/-1}
    {
        \cubelabel{\x-\inc}{\y+\inc}{0}{\textbf{T}}
    }
\foreach \x/\y in {-2/0,0/-2,2/0,0/2}
    {
        \cubelabel{\x-\inc}{\y+\inc}{0}{\rook}
    }
\foreach \x/\y in {1/-1,-1/1}
    {
        \cubelabel{\x-\inc}{\y+\inc}{0}{\textbf{F}}
    }
  \foreach \x/\y in {-2/0,-1/-1, -1/0,-1/1, 0/-2, 0/-1, 0/1 ,0/2,1/-1, 1/0,1/1,2/0}
    {
    \elemcube{\x+6-\inc}{\y+\inc}{0}
    }
\colourcubeZface[black]{-1+6-\inc}{-1+\inc}{0}
\colourcubeZface[black]{1+6-\inc}{1+\inc}{0}
\colourcubeZface[sangria]{-1+6-\inc}{1+\inc}{0}
\colourcubeZface[sangria]{1+6-\inc}{-1+\inc}{0}

\foreach \x/\y in {1/1,-1/-1}
    {
        \cubelabel{\x+6-\inc}{\y+\inc}{0}{\textbf{T}}
    }
\foreach \x/\y in {-2/0,0/-2,2/0,0/2}
    {
        \cubelabel{\x+6-\inc}{\y+\inc}{0}{\rook}
    }
\foreach \x/\y in {1/-1,-1/1}
    {
        \cubelabel{\x+6-\inc}{\y+\inc}{0}{\textbf{F}}
    }
  \foreach \x/\y in {-2/0,-1/-1, -1/0,-1/1, 0/-2, 0/-1, 0/1 ,0/2,1/-1, 1/0,1/1,2/0}
    {
    \elemcube{\x+12-\inc}{\y+\inc}{0}
    }
\colourcubeZface[black]{-1+12-\inc}{-1+\inc}{0}
\colourcubeZface[black]{1+12-\inc}{1+\inc}{0}
\colourcubeZface[sangria]{-1+12-\inc}{1+\inc}{0}
\colourcubeZface[sangria]{1+12-\inc}{-1+\inc}{0}

\foreach \x/\y in {1/1,-1/-1}
    {
        \cubelabel{\x+12-\inc}{\y+\inc}{0}{\textbf{T}}
    }
\foreach \x/\y in {-2/0,0/-2,2/0,0/2}
    {
        \cubelabel{\x+12-\inc}{\y+\inc}{0}{\rook}
    }
\foreach \x/\y in {1/-1,-1/1}
    {
        \cubelabel{\x+12-\inc}{\y+\inc}{0}{\textbf{F}}
    }
\foreach \x/\y/\z in {0/-1/1,6/-1/1,12/-1/1}
{\glasselemcube[forestgreen]{\x}{\y}{\z}{.8}}
\foreach \x/\y/\z in {0/-1/1,6/-1/1,12/-1/1}
    {
        \cubelabel{\x}{\y}{\z}{\rook}
    }
\foreach \x/\y/\z in {0/0/1}
{\glasselemcube[forestgreen]{\x}{\y}{\z}{.8}}
\glasselemcube[black]{0}{0}{2}{.9}
\colourcubeZface[black]{0}{0}{2}
\cubelabel{0}{0}{2}{\textbf{T}}

\foreach \x/\y/\z in {1/0/2,1/0/3,2/0/2,2/1/2,3/0/2,3/0/3,4/0/2,4/1/2,5/0/2,5/0/3,6/0/1}
{\glasselemcube[forestgreen]{\x}{\y}{\z}{.8}}
\glasselemcube[black]{6}{0}{2}{.9}
\colourcubeZface[black]{6}{0}{2}
\cubelabel{6}{0}{2}{\textbf{T}}

\foreach \x/\y/\z in {7/0/2,7/1/2,8/0/2,8/0/3,9/0/2,9/1/2,10/0/2,10/0/3,11/0/2,11/1/2,12/0/1 }
{\glasselemcube[forestgreen]{\x}{\y}{\z}{.8}}
\glasselemcube[black]{12}{0}{2}{.9}C
\colourcubeZface[black]{12}{0}{2}
\cubelabel{12}{0}{2}{\textbf{T}}
\foreach \x/\y/\z in {1/0/3,2/1/2,3/0/3,4/1/2,5/0/3,7/1/2,8/0/3,9/1/2,10/0/3,11/1/2 }
{\cubelabel{\x}{\y}{\z}{\rook}}
\draw (12,5,0) node[below]{$\mathbf{x_1}$};
\draw (6,5,0) node[below] {$\mathbf{x_2}$};
\draw  (0,5,0) node[below] {$\mathbf{x_3}$};
\end{tikzpicture}
\caption{The clause gadget of the clause  $\mathbf{x_1}\lor \mathbf{x_2}\lor \mathbf{x_3}$.\\ \url{https://skfb.ly/ozV8L}}\label{fig:clause-gadget-connection-rooks}
\end{figure}

\begin{proposition}\label{prop:psat-to-polyc-rooks}
From an instance $C$ of $\psat$, it is possible to construct a polycube $P_3^R(C)$ of polynomial size in polynomial time.
\end{proposition}
\begin{proof}
Let $C$ be an instance of $\psat$. We need to construct a polycube from this instance that is polynomial in size, and this construction needs to be done in polynomial time. We begin by constructing a polycube $P_3^R(C)$.

As $C$ is planar, we draw it as a visibility representation~\cite{DHLVM83} and then transfer this on a grid fitting in $O(n^2)$ squares via the procedure described by Biedl \emph{et al}~\cite{BBNNPR13}. The clauses and variables become squares, and these squares are joined by paths of squares of the grid if there is an edge between them. Figure~\ref{fig:path_graph-psat} presents a path grid version of the graph of the instance of Figure~\ref{fig:ex_psat}.

\begin{figure}[h]
    \centering
    \begin{tikzpicture}
        \foreach \x/\y in {0/0,0/1,0/2,0/3,0/4,0/5,0/6,0/7,0/8,
        1/0,1/4,1/8,
        2/0,2/1,2/2,2/3,2/4,2/8,
        3/0,3/4,3/8,
        4/0,4/4,4/5,4/5,4/6,4/7,4/8,
        5/0,5/8,
        6/0,6/1,6/2,6/3,6/4,6/5,6/6,6/7,6/8
        } {
          \path [draw=gray!, fill=brown!50] (.5+\x-0.45, .5+\y-1.45)
            -- ++(0,.9)
            -- ++(.9,0)
            -- ++(0,-.9)
            --cycle;
        }
        \foreach \x/\y in {2/0,2/4,2/8}
        {
        \path [draw=gray!, fill=teal!50] (.5+\x-0.45, .5+\y-1.45)
            -- ++(0,.9)
            -- ++(.9,0)
            -- ++(0,-.9)
            --cycle;
        }
          \foreach \x/\y in {0/4,2/2,4/6,6/4}
        {
        \path [draw=gray!, fill=le-purple!50] (.5+\x-0.45, .5+\y-1.45)
            -- ++(0,.9)
            -- ++(.9,0)
            -- ++(0,-.9)
            --cycle;
        }
        \node[anchor=west] at (.5+2-0.45, 0-.52) {$x_1$};
        \node[anchor=west] at (.5+2-0.45, 4-.52) {$x_2$};
        \node[anchor=west] at (.5+2-0.45, 8-.52) {$x_3$};
        \node[anchor=west] at (.5+0-0.45, 4-.52) {$c_1$};
        \node[anchor=west] at (.5+2-0.45, 2-.52) {$c_2$};
        \node[anchor=west] at (.5+4-0.45, 6-.52) {$c_3$};
        \node[anchor=west] at (.5+6-0.45, 4-.52) {$c_4$};
    \end{tikzpicture}
    \caption{The grid path graph of the instance $C$ of Figure~\ref{fig:ex_psat}. The clauses are in purple, the variables are in teal and the connection squares are in brown.}
    \label{fig:path_graph-psat}
\end{figure}

The gadgets used are not planar, so we will need to extend this grid to three dimensions. To fit the gadgets, every variable node would fit a $5\times 5\times 1$ grid, but to put a propagating connection gadget above or under, we replace the square by $5\times 5\times 9$  cubes. The squares corresponding to the edges will correspond to the propagating process described in Lemma~\ref{lem:connection-gadget-rooks}. On the squares corresponding to the edges in straight lines, the gadgets are contained in a $5\times 5\times 9$ cube grid: this corresponds to the  two ways of propagating the signal, under or above; see Figure~\ref{fig:clause-gadget-connection-rooks}. As they are placed on the variable gadgets to begin with, we replace all these squares by $4\times 5 \times 9$ cubes. On the corner squares, the gadgets have height 9 and so we have to replace the square by $9\times 9 \times 17$ cubes to allow both the above or under passes. Duplicating the signal, similarly, needs $9\times 13 \times 17$ cubes. See Figure~\ref{fig:connection-gadget-rooks}.

Only the clause squares remain. Once the signal is propagated to a clause square, place is needed to arrange the variable gadgets so that they are at the same level and parallel to each other. The last condition is there because clauses must not be put on a cube having a connection; this is an important condition for Lemma~\ref{lem:equivalent-rooks}, and we show the reason for this in Figure~\ref{fig:problem_clause_rook}. There also need some space in between the literals: 3 extra cubes for a clause with 2 literals and 6 extra cubes for a clause with 3 literals. Putting clause at the same level and parallel to each other can be done via 2 extra propagation of the signal after turning the corner. The clause gadget itself then has a height of 3 cubes. Hence replacing each clause square with $20\times 15\times 20$ cubes should leave enough space.

\begin{figure}[h]
    \centering
 \begin{tikzpicture}[x=(220:0.5cm), y=(-40:0.5cm), z=(90:0.353cm)
 ]
\foreach \x/\y/\z in {-2/0/0,-1/-1/0, -1/0/0,-1/1/0, 0/-2/0, 0/-1/0, 0/1/0 ,0/2/0,1/-1/0, 1/0/0,1/1/0,2/0/0}
    {
    \elemcube{\x}{\y}{\z}
    }
\elemcube{-1}{-3}{3}
\elemcube{1}{-3}{3}
\foreach \x/\y/\z in {-2/-2/2, -1/-2/1, -1/-1/1, -1/-2/2, -1/-2/3, 0/-2/1, 0/-2/3, 0/-2/4, 1/-2/1, 1/-1/1, 1/-2/2, 1/-2/3, 2/-2/2}
    {
    \elemcube{\x}{\y}{\z}
    }
    \foreach \x/\y/\z in {-2/0/0,-1/-1/0, -1/0/0,-1/1/0}
     {
     \elemcube{\x}{\y-4}{\z+4}
     }
\foreach \x/\y/\z in { 0/-2/0, 0/-1/0, 0/1/0 ,0/2/0,1/-1/0, 1/0/0,1/1/0,2/0/0}
    {
    \elemcube{\x}{\y-4}{\z+4}
    }
    \elemcube[sangria]{-1}{-1}{1}
    \elemcube[gray]{1}{-1}{1}
    \colourcubeZface[sangria]{1}{-5}{4}
 \colourcubeZface[gray]{-1}{-5}{4}
\foreach \x/\y/\z in {-2/-5/6,-2/-5/7,1/-6/5,1/-5/5,1/-5/6,0/-5/6,-1/-5/6,0/-5/7,-1/-4/6}
{
\glasselemcube[forestgreen]{\x}{\y}{\z}{.8}
 }
 \glasselemcube[gray]{1}{-5}{6}{.9}
  \colourcubeZface[sangria]{-1}{1}{0}
 \colourcubeZface[sangria]{-1}{-3}{4}
 \elemcube[gray]{1}{1}{0}
 \elemcube[gray]{1}{-2}{3}
 \elemcube[gray]{1}{-3}{4}
\end{tikzpicture}
\qquad \qquad
     \begin{tikzpicture}[scale=1]
        \foreach \x/\y in {1/5,2/5,3/5,4/5,4/4,5/1,5/2,5/3,5/4,5/5} {
            \path [draw=brown!, fill=brown!30] (.5+\x-0.45, .5+\y-1.45)
            -- ++(0,.9)
            -- ++(.9,0)
            -- ++(0,-.9)
            --cycle;
        }
        \foreach \x/\y in {1/5,3/5,5/1,5/3,5/5} {
            \node at (.5+\x, \y-.5) {\rook};
        }
          \foreach \x/\y in {2/7,4/5,5/4} {
          \path [draw=gray!, fill=gray!30] (.5+\x-0.45, .5+\y-1.45)
            -- ++(0,.9)
            -- ++(.9,0)
            -- ++(0,-.9)
            --cycle;
            \node[anchor=west] at (.5+\x-0.4, \y-.5) {\textbf{T}};
        }
          \foreach \x/\y in {5/2,4/4,2/5} {
          \path [draw=sangria!, fill=sangria!30] (.5+\x-0.45, .5+\y-1.45)
            -- ++(0,.9)
            -- ++(.9,0)
            -- ++(0,-.9)
            --cycle;
            \node[anchor=west] at (.5+\x-0.4, \y-.5) {\textbf{F}};
        }
        \foreach \x/\y in {1/6,2/6,2/8,3/7} {
          \path [draw=gray!, fill=forestgreen!30] (.5+\x-0.45, .5+\y-1.45)
            -- ++(0,.9)
            -- ++(.9,0)
            -- ++(0,-.9)
            --cycle;
        }
    \end{tikzpicture}
    \caption{A close-up of how to properly place the clause gadget. It is important to not put the clause gadget on top of a square coming from a connection gadget. The right view shows the side view and the reason for this: two T or two F cubes will be in the same column if the clause gadget was put on the top T cube.}
    \label{fig:problem_clause_rook}
\end{figure}

We recapitulate the construction. First, we take the instance $C$. We transform it into a visibility representation and then into a path grid. Since the highest height needed is 20, we enlarge this grid into a rectangular prism of height 20. Each clause is then enlarged: this also enlarges the connecting edge. The variables are then enlarged, and finally the edge connections are enlarged as needed. The polycube is then finally constructed starting from the variables placed at height 10 (in the middle of the prism) and propagating the signal to the clauses via connection gadgets going above or under, and duplicating gadgets. The clause gadgets are then put on the propagated literals on the free T or F cubes after placing them parallel and at the same height. This concludes the construction of the polycube $P_3^R(C)$.

 Since $C$ is an instance of $\psat$, there are $n$ variables,  $3n$ edges and maximally $9n$ clauses. The enlargement we did still keep the size of $P_3^R(C)$ polynomially bounded. Finally, this process is polynomial in time, since transferring the instance $C$ into a visibility representation and then into a path grid representation is doable in polynomial time~\cite{Tamassia86,BBNNPR13}. Indeed, the extra steps needed to place the actual polycube are doable in polynomial time by first propagating the signal starting from the variables to all the clauses and then following the described procedures at each clause.  This concludes the proof.
\end{proof}

\begin{remark}\label{rem:rook_improv}
    It is possible to reduce the size of the polycube constructed. For example, we can replace long connection edges by longer clauses.
\end{remark}
The process of Proposition~\ref{prop:psat-to-polyc-rooks} is exemplified (with the simplification of the previous remark) in Figure~\ref{fig:ex_polycube_rooks} by constructing the polycube $P_3^R(C)$ associated with the instance $C$ of Figure~\ref{fig:ex_psat} (compare its grid path graph in Figure~\ref{fig:path_graph-psat}).  The following lemma will prove a key property of the polycube $P_3^R(C)$.

\begin{figure}[h]
    \centering
        \begin{tikzpicture}[scale=.25]
    \foreach \x/\y in {
        2/1,3/1,4/1,5/1,6/1,7/1,8/1,
        3/2,
        3/3,4/3,
        3/5,4/5,
        3/6,
        2/7,3/7,4/7,5/7,6/7,
        5/9,6/9,
        5/10,
        2/11,3/11,4/11,5/11,6/11,7/11,8/11
        } {
                \foreach \z/\w in {0/0,-1/1,0/1,1/1,-2/2,-1/2,1/2,2/2,-1/3,0/3,1/3,0/4} {
                \path [draw=brown!, fill=brown!30] (.5+4*\x+\z-0.45, .5+4*\y+\w-1.45)
				   -- ++(0,.9)
				   -- ++(.9,0)
				   -- ++(0,-.9)
				   --cycle;
       }
       }
     \foreach \x/\y in {
        1/1,1/7, 1/11
        } {
                \foreach \z/\w in {0/0,-1/1,0/1,1/1,-2/2,-1/2,1/2,2/2,-1/3,0/3,1/3,0/4} {
                \path [draw=g-green!, fill=g-green!50] (.5+4*\x+\z-0.45, .5+4*\y+\w-1.45)
				   -- ++(0,.9)
				   -- ++(.9,0)
				   -- ++(0,-.9)
				   --cycle;
       }
       }
       \foreach \x/\y in {
        5/3,5/5
        } {
                \foreach \z/\w in {0/0,-1/1,0/1,1/1,-2/2,-1/2,1/2,2/2,-1/3,0/3,1/3,0/4} {
                \path [draw=le-pink!, fill=le-pink!50] (.5+4*\x+\z-0.45, .5+4*\y+\w-1.45)
				   -- ++(0,.9)
				   -- ++(.9,0)
				   -- ++(0,-.9)
				   --cycle;
       }
       }
       \foreach \x/\y in {
        7/7,7/9
        } {
                \foreach \z/\w in {0/0,-1/1,0/1,1/1,-2/2,-1/2,1/2,2/2,-1/3,0/3,1/3,0/4} {
                \path [draw=le-orange!, fill=le-orange!50] (.5+4*\x+\z-0.45, .5+4*\y+\w-1.45)
				   -- ++(0,.9)
				   -- ++(.9,0)
				   -- ++(0,-.9)
				   --cycle;
       }
       }
       \foreach \x/\y in {
        9/1,9/11
        } {
                \foreach \z/\w in {0/0,-1/1,0/1,1/1,-2/2,-1/2,1/2,2/2,-1/3,0/3,1/3,0/4} {
                \path [draw=le-purple!, fill=le-purple!50] (.5+4*\x+\z-0.45, .5+4*\y+\w-1.45)
				   -- ++(0,.9)
				   -- ++(.9,0)
				   -- ++(0,-.9)
				   --cycle;
       }
       }
       \foreach \x/\y in {1/2,1/3,1/4,1/5,1/6,1/8,1/9,1/10 
       }
       {\foreach \z/\w in {-1/1,0/1,-1/2,-1/3,0/3,-1/4} {
                \path [draw=gray!, fill=gray!30, opacity = .5] (.5+4*\x+\z-0.45, .5+4*\y+\w-1.45)
				   -- ++(0,.9)
				   -- ++(.9,0)
				   -- ++(0,-.9)
				   --cycle;
       }}
       \foreach \x/\y in {
       5/4, 
       7/8,
       9/2,9/3,9/4,9/5,9/6,9/7,9/8,9/9,9/10 
       }
       {\foreach \z/\w in {-1/1,-2/1,-1/2,-1/3,-2/3,-1/4} {
                \path [draw=gray!, fill=gray!30, opacity = .5] (.5+4*\x+2+\z-0.45, .5+4*\y+\w-1.45)
				   -- ++(0,.9)
				   -- ++(.9,0)
				   -- ++(0,-.9)
				   --cycle;
       }}
       \foreach \x/\y/\z/\w in {1/1/-1/4, 1/1/-2/3, 1/1/-1/3,
       1/7/-1/1, 1/7/-0/1, 1/7/-1/2, 1/7/-1/3, 1/7/-2/3, 1/7/-1/4,
       1/11/-1/1, 1/11/0/1,1/11/-1/2,1/11/-1/3,1/11/-2/3
       } {
                \path [draw=gray!, fill=gray!30, opacity = .5] (.5+4*\x+\z-0.45, .5+4*\y+\w-1.45)
				   -- ++(0,.9)
				   -- ++(.9,0)
				   -- ++(0,-.9)
				   --cycle;
       }
       \foreach \x/\y/\z/\w in {9/1/1/4, 9/1/1/3, 9/1/2/3,
       9/11/1/1,9/11/2/1,
       5/3/1/3, 5/3/2/3, 5/3/1/4,
       5/5/1/1, 5/5/0/1, 5/5/1/2, 5/5/1/3, 5/5/2/3,
       7/7/1/3, 7/7/2/3, 7/7/1/4,
       7/9/1/1, 7/9/0/1, 7/9/1/2, 7/9/1/3, 7/9/2/3
       } {
                \path [draw=gray!, fill=gray!30, opacity = .5] (.5+4*\x+\z-0.45, .5+4*\y+\w-1.45)
				   -- ++(0,.9)
				   -- ++(.9,0)
				   -- ++(0,-.9)
				   --cycle;
       }
       \draw (0,4*1+1) node {$x_3$};
       \draw (0,4*7+1) node {$x_2$};
       \draw (0,4*11+1) node {$x_1$};
       \draw (4*10+1,4*7+1) node[right,rotate=270] {$x_3  \lor \overline x_1$};
       \draw (4*6,4*4) node[rotate =270] {$\overline x_2 \lor \overline x_1$};
       \draw (4*9-3,4*8+1) node[rotate=270] {$  \overline x_3\lor \overline{x}_2$};
       \draw [decorate, decoration = {brace}] (-2,2) --  (-2,4*11+2);
       \draw (-6,4*7) node[rotate=90] {$x_1\lor x_2 \lor x_3$};
    \end{tikzpicture}
    \hfill \includegraphics[width=.5\textwidth]{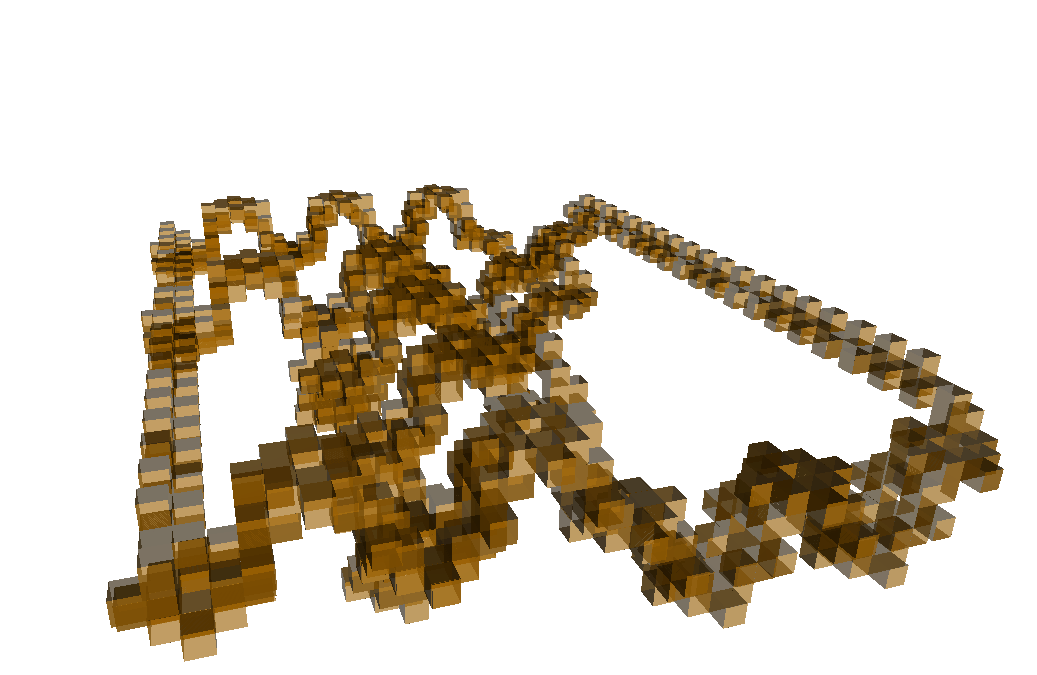}
    \caption{Left) Top view of the  polycube $P_3^R(C)$ given from Proposition~\ref{prop:psat-to-polyc-rooks} (with some simplifications) with the instance $C$ of Figure~\ref{fig:ex_psat} as input. In green, the three variables. In the remaining colours, the literals of the clauses. The brown tiles are the top view of the connection construction. Right) A 3D rendering of the model.}
    \label{fig:ex_polycube_rooks}
\end{figure}


\begin{lemma}\label{lem:equivalent-rooks}
Let $C$ be an instance of $\psat$, and let $P_3^R(C)$ be its associated polycube. Let
\begin{equation}
m_R := 6 n_{var} + 6 n_{connect} + \sum_{c \in\{\text{clauses of }C\}} \ell_R(c),
\end{equation}
where $n_{var}$ is the number of variables, $n_{connect}$ is the total number of connection gadgets, $\ell_R(c)$ is the length of the clause $c$, and $n_{clause}$ is the total number of clauses. There exists a non-attacking rook set of size $m_R + n_{clause}$ for the polycube $P_3^R(C)$ if and only if $C$ is satisfiable.
\end{lemma}

\begin{proof}
    Consider a truth assignment that satisfies all the clauses in $C$. We transfer the assignment on the polycube $P_3^R(C)$ of Proposition~\ref{prop:psat-to-polyc-rooks}. Each variable can have at most 6 rooks placed on them regardless of their value, so there are 6 rooks per variable gadget.

    We connect the variable gadgets to the clause gadgets. Each of them have a maximum number of 6 rooks regardless of the signal they transfer, so there can be 6 rooks per connection gadget.

    Since it is a truth assignment, all clauses are satisfied.  Lemma~\ref{lem:clause-gadget-rooks} implies that  one additional non-attacking rook per clause can be placed.

    The total number of non-attacking rooks that can be placed on $P_3^R(C)$ is then $m_R + n_{clause}$, proving that an instance $C$ is transformed into an independent rook domination problem $(P_3^R(C),m_R+n_{clause})_3^R$.

    If a set of $m_R + n_{clause}$ non-attacking rooks dominating $P_3^R(C)$ exists,  and since it is not possible to increase the number of rooks in either the variable or in the connection gadgets, there must be $\ell_R(c)+1$ rooks placed on each clause $c$ gadget. From Lemma~\ref{lem:clause-gadget-rooks}, it follows that all the corresponding clauses evaluate to \true, meaning that $C$ is satisfied.
\end{proof}

We can now prove Theorem~\ref{thm:nphard-rooks}.
\begin{proof}[Proof of Theorem~\ref{thm:nphard-rooks}]
Lemma~\ref{lem:verifP} implies that the problem is in the class NP since verifying a solution can be done in polynomial time.

Let $C$ be an instance of $\psat$. We obtain a polynomially-sized polycube $P_3^R(C)$ in polynomial time by Proposition~\ref{prop:psat-to-polyc-rooks}. Lemma~\ref{lem:equivalent-rooks} implies that the non-attacking rook set problem for $(P_3^R(C), m_R+n_{clause})_3^R$ is equivalent to the $\psat$ problem for $C$, which is NP-complete by Proposition~\ref{prop:psatNP}, and so the non-attacking rook set problem for $3$--polycube is NP-complete. As $3$--polycubes are contained in $d$--polycubes for $d\geq 3$, the non-attacking rook set problem for $d$--polycubes is also NP-complete.
\end{proof}

We add a last note before closing the section. Two different instances can lead to the same polycube. For example, the instance constructed from a given instance by switching every literal $x$ to $\overline{x}$ and every literal $\overline{x}$ to $x$ in the clauses will yield the same polycube: it amounts to interchanging the T and F cubes. So, we should expect two solutions for each satisfying choice of boolean variables satisfying the clauses.

\subsection{NP-completion of maximum independent domination of queens on polycubes}

We now turn to the proof of Theorem~\ref{thm:nphard-queens}. It will proceed in a similar fashion as the previous section, with different gadgets. The first component of the proof is the variable gadget obtained as a concatenation of four smaller elements identical to the rook variable gadgets; see Figure~\ref{fig:variable-gadget-queens_3D_construction}.

\begin{figure}[h]
\begin{subfigure}{.3\textwidth}
    \centering
    \begin{tikzpicture}[x=(200:0.5cm), y=(-40:0.5cm), z=(90:0.353cm)]
		  \foreach \x/\y in {1/2, 2/1, 2/2, 2/3,3/0, 3/1, 3/3, 3/4, 4/1, 4/2, 4/3, 5/2}
			{
			\elemcube{\x-\inc}{\y+\inc}{0}
			}

		\foreach \x/\y in {2/2, 3/0, 4/2, 3/4}
			{
				\colourcubeZface[black]{\x-\inc}{\y+\inc}{0}
			}
		\foreach \x/\y in {2/2, 3/0, 4/2, 3/4}
			{
				\cubelabel{\x-\inc}{\y+\inc}{0}{\textbf{T}}
			}
		\foreach \x/\y in {1/2, 3/1, 3/3, 5/2}
			{
				\colourcubeZface[sangria]{\x-\inc}{\y+\inc}{0}
			}
		\foreach \x/\y in {1/2, 3/1, 3/3, 5/2}
			{
				\cubelabel{\x-\inc}{\y+\inc}{0}{\textbf{F}}
			}
		\end{tikzpicture}
	   \caption{One element.}
	   \label{fig:element-queens}
\end{subfigure}
	\begin{subfigure}{.69\textwidth}
	    \centering
	\begin{tikzpicture}[x=(200:0.5cm), y=(-40:0.5cm), z=(90:0.353cm)]
		  \foreach \x/\y in {1/2, 1/6, 2/1, 2/2, 2/3, 2/5, 2/6, 2/7, 3/0, 3/1, 3/3, 3/4, 3/5, 3/7, 3/8, 4/1, 4/2, 4/3, 4/5, 4/6, 4/7, 5/2, 5/6, 6/1, 6/2, 6/3, 6/5, 6/6, 6/7, 7/0, 7/1, 7/3, 7/4, 7/5, 7/7, 7/8, 8/1, 8/2, 8/3, 8/5, 8/6, 8/7, 9/2, 9/6}
			{
			\elemcube{\x-\inc}{\y+\inc}{0}
			}
		\foreach \x/\y in {1/6, 2/2, 3/0, 4/2, 7/1, 7/3, 3/5, 3/7, 6/6, 8/6, 9/2, 7/8}
			{
				\colourcubeZface[black]{\x-\inc}{\y+\inc}{0}
			}
		\foreach \x/\y in {1/6, 2/2, 3/0, 4/2, 7/1, 7/3, 3/5, 3/7, 6/6, 8/6, 9/2, 7/8}
			{
				\cubelabel{\x-\inc}{\y+\inc}{0}{\textbf{T}}
			}
		\foreach \x/\y in {1/2, 3/1, 3/3, 6/2, 7/0, 8/2, 2/6, 4/6, 3/8, 7/5, 7/7, 9/6}
			{
				\colourcubeZface[sangria]{\x-\inc}{\y+\inc}{0}
			}
		\foreach \x/\y in {1/2, 3/1, 3/3, 6/2, 7/0, 8/2, 2/6, 4/6, 3/8, 7/5, 7/7, 9/6}
			{
				\cubelabel{\x-\inc}{\y+\inc}{0}{\textbf{F}}
			}
		\end{tikzpicture}
	   \caption{
  Variable gadget for 3D queens from four elements.}
	   \label{fig:variable-gadget-queens_3D}
	\end{subfigure}
 \caption{Construction of the variable 
 gadget for 3D queens. Maximally 12 queens can be placed in two different ways.}\label{fig:variable-gadget-queens_3D_construction}
\end{figure}

\begin{lemma}\label{lem:literal-gadget-queens_3D}
    The gadget of Figure~\ref{fig:variable-gadget-queens_3D} has two states of maximum domination by 12 queens.
\end{lemma}
\begin{proof}
    We remark that the variable gadget, albeit made from 3-dimensional cubes, is inherently 2-dimensional. Hence, we check that there are only two placements of 12 queens by the solver~\cite{GadgetCheck} introduced in Tool~\ref{tool:gadget-check} and that this is the maximum. Alternatively, we can proceed by brute force, since the gadget is relatively small.
\end{proof}

Contrary to the rook gadgets, the number of queens will not remain constant on the connected variable gadget. This means that we must carefully analyze the propagation of the signal via the connection gadget presented in Figure~\ref{fig:connection-gadget-queens-3D}.

 \begin{figure}[h]
		\centering
		\begin{tikzpicture}[scale=1]

      \foreach \x/\y in {
        -1/0,-1/4,3/0,3/4,
        7/0,7/4,7/8,11/0,11/4,11/8,
        15/0,15/4
        } {
                \foreach \z/\w in {0/0,-1/1,0/1,1/1,-2/2,-1/2,1/2,2/2,-1/3,0/3,1/3,0/4} {
                \path [draw=brown!, fill=brown!30] (.5+\x+\z-0.45, .5+\y+\w-1.45)
				   -- ++(0,.9)
				   -- ++(.9,0)
				   -- ++(0,-.9)
				   --cycle;
       }
       }
				\foreach \x/\y in {-2/6, -1/8, -1/3, -1/1, -3/2, 0/6, 2/2, 3/0, 4/2, 7/1, 7/3, 3/5, 3/7, 6/6, 8/6, 10/2, 12/2, 11/0, 7/9, 7/11, 5/10, 11/5, 11/7, 10/10, 12/10, 11/12, 14/6, 16/6, 15/8, 15/3, 15/1, 17/2} {
				 \path [draw=gray!, fill=gray!30] (.5+\x-0.45, .5+\y-1.45)
				   -- ++(0,.9)
				   -- ++(.9,0)
				   -- ++(0,-.9)
				   --cycle;
				   \node[anchor=west] at (.5+\x-0.4, \y-.5) {\textbf{T}};
			   }
				 \foreach \x/\y in {-2/2,-1/0, -1/5, -1/7, -3/6, 0/2, 3/1, 3/3, 6/2, 7/0, 8/2, 2/6, 4/6, 3/8, 7/5, 7/7, 10/6, 12/6, 11/1, 11/3, 6/10, 8/10, 7/12, 11/9, 11/11, 13/10, 14/2, 15/0, 16/2, 15/5, 15/7, 17/6} {
				 \path [draw=sangria!, fill=sangria!30] (.5+\x-0.45, .5+\y-1.45)
				   -- ++(0,.9)
				   -- ++(.9,0)
				   -- ++(0,-.9)
				   --cycle;

				   \node[anchor=west] at (.5+\x-0.4, \y-.5) {\textbf{F}};
			   }
			   \draw (9.5,12.5) node[align=center]{$\uparrow$};
			   \draw (18.5,3.5) node[align=center]{$\rightarrow$};
			   \draw (1.5,3.5) node[align=center]{Lit.};
      \draw[thick,dotted, teal]
      (6, -1 -.2) -- (6,9-.15) -- (5-.2,9-.15) -- (5-.2,12+.1) -- (18+.2,12+.1) -- (18+.2,-1-.2) -- cycle;
      \draw[thick,orange]
      (6-.1,-1-.2) -- (6-.1,8+.1) --(-3-.1,8+.1) --(-3-.1,-1-.2) -- cycle;
      \draw[thick,dashed]
      ( (6-.1,8+.15) -- (14-.1,8+.15)--(14,4-.1)--(6-.1,4-.1)--cycle;
		   \end{tikzpicture}
		   \caption{Top view of the connection gadget (in the dotted teal box) propagating the signal of a variable gadget (in the full orange box). There is one pair of elements, the two in a dashed black box, that  have each four elements as neighbors.}\label{fig:connection-gadget-queens-3D}
	\end{figure}

\begin{lemma}\label{lem:propagation-gadget-queens-3D}
The signal of the variable gadget is propagated by the connection gadget of Figure~\ref{fig:connection-gadget-queens-3D}. The number of queens needed to guard the gadget is given by $4n_{4neigh} + 5n_{3neigh} + 6n_{2neigh}$, where $n_{ineigh}$ is the number of pairs of tiles with $i$ neighbors. Furthermore, the connection gadget can be used to turn corners.
\end{lemma}
\begin{proof}
We begin by counting the number of queens on the gadgets. We first notice that the basic elements come in pairs by construction. Both members of the pair can have either 2, 3, or 4 neighbors. Then, a local analysis shows that the pair will have maximally 6, 5, or 4 queens, respectively. The three situations can be seen in Figure~\ref{fig:connection-gadget-queens-3D}, where the one pair with 4 neighbors has been highlighted.
\end{proof}
Finally, to evaluate clauses, gadgets similar to those of Figures~\ref{fig:clause-2gadget-rooks} and~\ref{fig:clause-gadget-connection-rooks}
are used. Their construction proceeds as follows. Let $c$ be a clause with three literals (respectively two). Suppose there are three  variable gadgets, $x_1,x_2,x_3$ (respectively two, $x_1,x_2$) on a line. Then the extremal nodes form a sequence of red (\false) and gray (\true) tiles. If the variable $x_i$ appears as $x_i$ in $C$, we place a connector \begin{tikzpicture}[x=(200:0.5cm), y=(-40:0.5cm), z=(90:0.353cm),baseline= {(current bounding box.center)}, scale = .25]
\foreach \x/\y/\z in {0/-2/0,0/-1/0, 0/0/0,0/0/1}
{\elemcube[forestgreen]{\x}{\y}{\z}}
\end{tikzpicture}
on the red (F) tile, and if it appears as $\overline{x}_i$ in $C$, we place the connector on the gray (T) tile. We then join the connector by alternating top
\begin{tikzpicture}[x=(200:0.5cm), y=(-40:0.5cm), z=(90:0.353cm),baseline= {(current bounding box.center)}, scale = .25]
\foreach \x/\y/\z in {0/0/1,0/0/2,0/0/3}
{\elemcube[forestgreen]{\x}{\y}{\z}}
\end{tikzpicture}
and left
\begin{tikzpicture}[x=(200:0.5cm), y=(-40:0.5cm), z=(90:0.353cm),baseline= {(current bounding box.center)}, scale = .25]
\foreach \x/\y/\z in {0/0/1,0/1/1,0/2/1}
{\elemcube[forestgreen]{\x}{\y}{\z}}
\end{tikzpicture}
nodes. They form a clause gadget---see Figure~\ref{fig:clause-2gadget-connection-queens-3D}. The length of the clause gadget is denoted as $\ell_Q(c)$ and is defined similarly as the length $\ell_R$ of the rook gadget in~\eqref{eq:f_clause}. The process is illustrated in Figure~\ref{fig:procedure_clause}.
%
%
\begin{figure}[h]
\[
\begin{tikzpicture}[scale = .25,baseline= {(current bounding box.center)}]
 \foreach \x/\y in {
        0/5,0/6,1/5,1/6,3/5,3/6,4/5,4/6,6/5,6/6,7/5,7/6,6/2,6/3,7/2,7/3
        } {
                \foreach \z/\w in {0/0,-1/1,0/1,1/1,-2/2,-1/2,1/2,2/2,-1/3,0/3,1/3,0/4} {
                \path [draw=brown!, fill=brown!30] (.5+4*\x+\z-0.45, .5+4*\y+\w-1.45)
				   -- ++(0,.9)
				   -- ++(.9,0)
				   -- ++(0,-.9)
				   --cycle;
       }
       }
\foreach \x/\y in { 0/6,3/6,6/6, 6/3}
{\foreach \z/\w in {0/4} {
                \path [draw=red!, fill=red!50] (.5+4*\x+\z-0.45, .5+4*\y+\w-1.45)
				   -- ++(0,.9)
				   -- ++(.9,0)
				   -- ++(0,-.9)
				   --cycle;
}}
\foreach \x/\y in { 1/6,4/6,7/6,7/3}
{\foreach \z/\w in {0/4} {
                \path [draw=gray!, fill=gray!50] (.5+4*\x+\z-0.45, .5+4*\y+\w-1.45)
				   -- ++(0,.9)
				   -- ++(.9,0)
				   -- ++(0,-.9)
				   --cycle;
}}
\foreach \x/\y in { 7/6, 7/3}
{\foreach \z/\w in {2/2} {
                \path [draw=red!, fill=red!50] (.5+4*\x+\z-0.45, .5+4*\y+\w-1.45)
				   -- ++(0,.9)
				   -- ++(.9,0)
				   -- ++(0,-.9)
				   --cycle;
}}
\foreach \x/\y in { 7/5,7/2}
{\foreach \z/\w in {2/2} {
                \path [draw=gray!, fill=gray!50] (.5+4*\x+\z-0.45, .5+4*\y+\w-1.45)
				   -- ++(0,.9)
				   -- ++(.9,0)
				   -- ++(0,-.9)
				   --cycle;
}}
\end{tikzpicture} \xrightarrow{c_1 = x_1\lor x_2\lor x_3,\ c_2 = \overline{x}_3 \lor \overline{x}_4}
\begin{tikzpicture}[scale = .25,baseline= {(current bounding box.center)}]
 \foreach \x/\y in {
        0/5,0/6,1/5,1/6,3/5,3/6,4/5,4/6,6/5,6/6,7/5,7/6,6/2,6/3,7/2,7/3
        } {
                \foreach \z/\w in {0/0,-1/1,0/1,1/1,-2/2,-1/2,1/2,2/2,-1/3,0/3,1/3,0/4} {
                \path [draw=brown!, fill=brown!30] (.5+4*\x+\z-0.45, .5+4*\y+\w-1.45)
				   -- ++(0,.9)
				   -- ++(.9,0)
				   -- ++(0,-.9)
				   --cycle;
       }
       }
\foreach \x/\y in { 0/6,3/6,6/6, 6/3}
{\foreach \z/\w in {0/4} {
                \path [draw=red!, fill=red!50] (.5+4*\x+\z-0.45, .5+4*\y+\w-1.45)
				   -- ++(0,.9)
				   -- ++(.9,0)
				   -- ++(0,-.9)
				   --cycle;
}}
\foreach \x/\y in { 1/6,4/6,7/6,7/3}
{\foreach \z/\w in {0/4} {
                \path [draw=gray!, fill=gray!50] (.5+4*\x+\z-0.45, .5+4*\y+\w-1.45)
				   -- ++(0,.9)
				   -- ++(.9,0)
				   -- ++(0,-.9)
				   --cycle;
}}
\foreach \x/\y in { 7/6, 7/3}
{\foreach \z/\w in {2/2} {
                \path [draw=red!, fill=red!50] (.5+4*\x+\z-0.45, .5+4*\y+\w-1.45)
				   -- ++(0,.9)
				   -- ++(.9,0)
				   -- ++(0,-.9)
				   --cycle;
}}
\foreach \x/\y in { 7/5,7/2}
{\foreach \z/\w in {2/2} {
                \path [draw=gray!, fill=gray!50] (.5+4*\x+\z-0.45, .5+4*\y+\w-1.45)
				   -- ++(0,.9)
				   -- ++(.9,0)
				   -- ++(0,-.9)
				   --cycle;
}}
\foreach \x in {
        1,2,4,5
        } {
                \foreach \z/\w in {-2/4,-1/4,-1/3,-1/2,0/4,1/4,1/3,1/2} {
                \path [draw=teal!, fill=teal!50,opacity=.4] (.5+4*\x+\z-0.45, .5+4*6+\w-1.45)
				   -- ++(0,.9)
				   -- ++(.9,0)
				   -- ++(0,-.9)
				   --cycle;
       }
       }
\foreach \z/\w in {0/4,0/5,0/6,-2/4,-1/4,-1/3,-1/2} {
                \path [draw=teal!, fill=teal!50,opacity = .4] (.5+4*6+\z-0.45, .5+4*6+\w-1.45)
				   -- ++(0,.9)
				   -- ++(.9,0)
				   -- ++(0,-.9)
				   --cycle;
}
\foreach \z/\w in {0/4,0/5,0/6,-2/4,-1/4,-1/3,-1/2,1/4,1/3,1/2} {
                \path [draw=teal!, fill=teal!50,opacity = .4] (.5+4*3+\z-0.45, .5+4*6+\w-1.45)
				   -- ++(0,.9)
				   -- ++(.9,0)
				   -- ++(0,-.9)
				   --cycle;
}
\foreach \z/\w in {0/4,0/5,0/6,1/4,1/3,1/2} {
                \path [draw=teal!, fill=teal!50,opacity = .4] (.5+4*0+\z-0.45, .5+4*6+\w-1.45)
				   -- ++(0,.9)
				   -- ++(.9,0)
				   -- ++(0,-.9)
				   --cycle;
}
\foreach \y in {
        3,4
        } {
                \foreach \z/\w in {2/0,2/1,1/1,0/1,2/2,2/3,1/3,0/3} {
                \path [draw=teal!, fill=teal!50,opacity=.4] (.5+4*7+\z-0.45, .5+4*\y+\w-1.45)
				   -- ++(0,.9)
				   -- ++(.9,0)
				   -- ++(0,-.9)
				   --cycle;
       }
       }
\foreach \z/\w in {2/2,3/2,4/2,2/3,1/3,0/3} {
                \path [draw=teal!, fill=teal!50,opacity = .4] (.5+4*7+\z-0.45, .5+4*2+\w-1.45)
				   -- ++(0,.9)
				   -- ++(.9,0)
				   -- ++(0,-.9)
				   --cycle;
}
\foreach \z/\w in {2/2,3/2,4/2,2/1,1/1,0/1,2/0} {
                \path [draw=teal!, fill=teal!50,opacity = .4] (.5+4*7+\z-0.45, .5+4*5+\w-1.45)
				   -- ++(0,.9)
				   -- ++(.9,0)
				   -- ++(0,-.9)
				   --cycle;
}
\end{tikzpicture}
\]
\caption{Top view of the procedure to construct 2 clause gadgets on 4 variable gadgets, $x_1,x_2,x_3,x_4$. The clause gadgets are in teal and are on top of the brown polycubes.}\label{fig:procedure_clause}
\end{figure}

\begin{figure}[h]
	\begin{tikzpicture}[rotate = -50,x=(245:0.5cm), y=(-30:0.5cm), z=(90:0.353cm)
	]
	\foreach \x/\y/\z in {2/-2/1,2/-1/1,12/-2/1,12/-1/1,2/0/1}
	{\elemcube[forestgreen]{\x}{\y}{\z}}
	\foreach \x/\y/\z in {2/-2/1,12/-2/1}
		{
			\cubelabel{\x}{\y}{\z}{\queen}
		}
	\elemcube[black]{2}{0}{2}
	\colourcubeZface[black]{2}{0}{2}
	\cubelabel{2}{0}{2}{\textbf{T}}
	\foreach \x/\y/\z in {3/0/2,3/0/3,3/0/4,4/0/2,4/1/2,4/2/2,5/0/2,5/0/3,5/0/4,6/0/2,6/1/2,6/2/2,7/0/2,7/0/3,7/0/4,8/0/2,8/1/2,8/2/2,9/0/2,9/0/3,9/0/4,10/0/2,10/1/2,10/2/2,11/0/2,11/0/3,11/0/4,12/0/1 }
	{\elemcube[forestgreen]{\x}{\y}{\z}}
	\elemcube[black]{12}{0}{2}
	\colourcubeZface[black]{12}{0}{2}
	\cubelabel{12}{0}{2}{\textbf{T}}

	\foreach \x/\y/\z in {3/0/4,4/2/2,5/0/4,6/2/2,7/0/4,8/2/2,9/0/4,10/2/2,11/0/4}
	{\cubelabel{\x}{\y}{\z}{\queen}}
	\end{tikzpicture}
	\caption{The 3D queen clause gadget for the clauses $x_1 \lor x_2$ if the variable gadgets are at distance one from each other.}\label{fig:clause-2gadget-connection-queens-3D}
\end{figure}

\def\offX{1}
\def\offY{-1}
\def\offV{10}
\begin{figure}[h]
	\centering
	\begin{tikzpicture}[x=(200:0.5cm), y=(-40:0.5cm), z=(90:0.353cm)]
	  \foreach \x/\y in {-2/0,-1/-1,-1/0,-1/1, 0/-2, 0/-1, 0/1 ,0/2,1/-1, 1/0,1/1,2/0,3/-1/,3/0,3/1,4/-2,4/-1,4/1,4/2,5/-1,5/0,5/1,6/0}
		{
		\elemcube{\x-\inc}{\y+\inc}{0}
		}
	\colourcubeZface[black]{0-\inc}{-2+\inc}{0}
	\colourcubeZface[black]{-1-\inc}{0+\inc}{0}
    \colourcubeZface[black]{1-\inc}{0+\inc}{0}
    \colourcubeZface[black]{4-\inc}{-1+\inc}{0}
    \colourcubeZface[black]{4-\inc}{1+\inc}{0}
    \colourcubeZface[black]{6-\inc}{0+\inc}{0}
	\colourcubeZface[sangria]{-2-\inc}{0+\inc}{0}
	\colourcubeZface[sangria]{0-\inc}{-1+\inc}{0}
    \colourcubeZface[sangria]{0-\inc}{1+\inc}{0}
    \colourcubeZface[sangria]{3-\inc}{0+\inc}{0}
    \colourcubeZface[sangria]{5-\inc}{0+\inc}{0}
    \colourcubeZface[sangria]{4-\inc}{-2+\inc}{0}
	\foreach \x/\y in {0/-2,-1/-0,1/0,4/-1,4/1,6/0}
		{
			\cubelabel{\x-\inc}{\y+\inc}{0}{\textbf{T}}
		}
	\foreach \x/\y in {-2/0,0/-1,0/1,3/0,5/0,4/-2}
		{
			\cubelabel{\x-\inc}{\y+\inc}{0}{\textbf{F}}
		}
    \foreach \x/\y in {-2/0,-1/-1,-1/0,-1/1, 0/-2, 0/-1, 0/1 ,0/2,1/-1, 1/0,1/1,2/0,3/-1/,3/0,3/1,4/-2,4/-1,4/1,4/2,5/-1,5/0,5/1,6/0}
		{
		\elemcube{\x+\offV-\inc}{\y+\inc}{0}
		}
    \colourcubeZface[black]{0+\offV-\inc}{-2+\inc}{0}
    \colourcubeZface[black]{-1+\offV-\inc}{0+\inc}{0}
    \colourcubeZface[black]{1+\offV-\inc}{0+\inc}{0}
    \colourcubeZface[black]{4+\offV-\inc}{-1+\inc}{0}
    \colourcubeZface[black]{4+\offV-\inc}{1+\inc}{0}
    \colourcubeZface[black]{6+\offV-\inc}{0+\inc}{0}
    \colourcubeZface[sangria]{-2+\offV-\inc}{0+\inc}{0}
    \colourcubeZface[sangria]{0+\offV-\inc}{-1+\inc}{0}
    \colourcubeZface[sangria]{0+\offV-\inc}{1+\inc}{0}
    \colourcubeZface[sangria]{3+\offV-\inc}{0+\inc}{0}
    \colourcubeZface[sangria]{5+\offV-\inc}{0+\inc}{0}
    \colourcubeZface[sangria]{4+\offV-\inc}{-2+\inc}{0}

	\foreach \x/\y in {0/-2,-1/-0,1/0,4/-1,4/1,6/0}
		{
			\cubelabel{\x+\offV-\inc}{\y+\inc}{0}{\textbf{T}}
		}
	\foreach \x/\y in {-2/0,0/-1,0/1,3/0,5/0,4/-2}
		{
			\cubelabel{\x+\offV-\inc}{\y+\inc}{0}{\textbf{F}}
		}
	\foreach \x/\y/\z in {2/-2/1,2/-1/1,12/-2/1,12/-1/1,2/0/1}
    {\glasselemcube[forestgreen]{\x+\offX}{\y+\offY}{\z}{.8}}
    \foreach \x/\y/\z in {2/-2/1,12/-2/1}
    {
        \cubelabel{\x+\offX}{\y+\offY}{\z}{\queen}
    }
    \glasselemcube[black]{2+\offX}{0+\offY}{2}{.9}
    \cubelabel{2+\offX}{0+\offY}{2}{\textbf{T}}
    \foreach \x/\y/\z in {3/0/2,3/0/3,3/0/4,4/0/2,4/1/2,4/2/2,5/0/2,5/0/3,5/0/4,6/0/2,6/1/2,6/2/2,7/0/2,7/0/3,7/0/4,8/0/2,8/1/2,8/2/2,9/0/2,9/0/3,9/0/4,10/0/2,10/1/2,10/2/2,11/0/2,11/0/3,11/0/4,12/0/1 }
    {\glasselemcube[forestgreen]{\x+\offX}{\y+\offY}{\z}{.8}}
    \glasselemcube[black]{12+\offX}{0+\offY}{2}{.9}
    \colourcubeZface[black]{12+\offX}{0+\offY}{2}
    \cubelabel{12+\offX}{0+\offY}{2}{\textbf{T}}

    \foreach \x/\y/\z in {3/0/4,4/2/2,5/0/4,6/2/2,7/0/4,8/2/2,9/0/4,10/2/2,11/0/4}
    {\cubelabel{\x+\offX}{\y+\offY}{\z}{\queen}}
  \foreach \x/\y in {-2/0,-1/-1,-1/0,-1/1, 0/-2, 0/-1, 0/1 ,0/2,1/-1, 1/0,1/1,2/0,3/-1/,3/0,3/1,4/-2,4/-1,4/1,4/2,5/-1,5/0,5/1,6/0}
		{
		\elemcube{\x-\inc}{\y+\inc+4}{0}
		}
  \colourcubeZface[sangria]{0-\inc}{2+\inc+4}{0}
    \colourcubeZface[sangria]{-1+-\inc}{0+\inc+4}{0}
    \colourcubeZface[sangria]{1-\inc}{0+\inc+4}{0}
    \colourcubeZface[sangria]{4-\inc}{-1+\inc+4}{0}
    \colourcubeZface[sangria]{4-\inc}{1+\inc+4}{0}
    \colourcubeZface[sangria]{6-\inc}{0+\inc+4}{0}
    \colourcubeZface[black]{-2-\inc}{0+\inc+4}{0}
    \colourcubeZface[black]{0-\inc}{-1+\inc+4}{0}
    \colourcubeZface[black]{0+-\inc}{1+\inc+4}{0}
    \colourcubeZface[black]{3-\inc}{0+\inc+4}{0}
    \colourcubeZface[black]{5-\inc}{0+\inc+4}{0}
    \colourcubeZface[black]{4-\inc}{2+\inc+4}{0}

	\foreach \x/\y in {0/2,-1/-0,1/0,4/-1,4/1,6/0}
		{
			\cubelabel{\x-\inc}{\y+\inc+4}{0}{\textbf{T}}
		}
	\foreach \x/\y in {-2/0,0/-1,0/1,3/0,5/0,4/2}
		{
			\cubelabel{\x-\inc}{\y+\inc+4}{0}{\textbf{F}}
		}
  \foreach \x/\y in {-2/0,-1/-1,-1/0,-1/1, 0/-2, 0/-1, 0/1 ,0/2,1/-1, 1/0,1/1,2/0,3/-1/,3/0,3/1,4/-2,4/-1,4/1,4/2,5/-1,5/0,5/1,6/0}
		{
		\elemcube{\x-\inc+10}{\y+\inc+4}{0}
		}
  \colourcubeZface[sangria]{0-\inc+\offV}{2+\inc+4}{0}
	\colourcubeZface[sangria]{-1-\inc+\offV}{0+\inc+4}{0}
    \colourcubeZface[sangria]{1-\inc+\offV}{0+\inc+4}{0}
    \colourcubeZface[sangria]{4-\inc+\offV}{-1+\inc+4}{0}
    \colourcubeZface[sangria]{4-\inc+\offV}{1+\inc+4}{0}
    \colourcubeZface[sangria]{6-\inc+\offV}{0+\inc+4}{0}
	\colourcubeZface[black]{-2-\inc+\offV}{0+\inc+4}{0}
	\colourcubeZface[black]{0-\inc+\offV}{-1+\inc+4}{0}
    \colourcubeZface[black]{0-\inc+\offV}{1+\inc+4}{0}
    \colourcubeZface[black]{3-\inc+\offV}{0+\inc+4}{0}
    \colourcubeZface[black]{5-\inc+\offV}{0+\inc+4}{0}
    \colourcubeZface[black]{4-\inc+\offV}{2+\inc+4}{0}

	\foreach \x/\y in {0/2,-1/-0,1/0,4/-1,4/1,6/0}
		{
			\cubelabel{\x-\inc+10}{\y+\inc+4}{0}{\textbf{T}}
		}
	\foreach \x/\y in {-2/0,0/-1,0/1,3/0,5/0,4/2}
		{
			\cubelabel{\x-\inc+10}{\y+\inc+4}{0}{\textbf{F}}
		}
	\draw (11.35,9.45,0) node[below]{$\mathbf{x_1}$};
	\draw  (1.35,9.45,0) node[below] {$\mathbf{x_2}$};
	\end{tikzpicture}
	\caption{The clause gadget on top of the two variable gadgets for the clause $\mathbf x_1\lor \mathbf x_2$.}\label{fig:clause-gadget-clausel-queens-3D}
\end{figure}

\begin{lemma}\label{lem:clause-gadgets-queens}
    Let $c$ be a clause. The clause gadget associated with $c$ when placed on top of variable gadgets can take $\ell_Q(c)+1$ queens if it evaluates to \true{}, or $\ell_Q(c)$ if it evaluates to \false{}, where the function $\ell_Q$ is the length of the clause gadget.
\end{lemma}
\begin{proof}
    Let $c$ be a clause and $K(c)$ be the associated clause gadget. The clause gadget is joined with its literals  placed on a line by placing it on top of the extremal nodes; see the example of Figure~\ref{fig:clause-gadget-clausel-queens-3D}. It covers the leftmost to the rightmost variable and is placed on F tiles if $x_i$ appears as $x_i$ in $c$, and on T tiles if it appears as $\overline x_i$ in $c$. Then, a maximum placement of queens will have all queens on the protruding nodes and can also have a queen on the gray T tiles if the clause is satisfied. There is one queen for each protruding node, one more for each connector, and maximally one extra in one of the connectors, thus showing that there are $\ell_Q(c)$ or $\ell_Q(c)+1$ queens in a maximum dominating position.
\end{proof}

\begin{proposition}\label{prop:psat-to-polyc-queens}
    From an instance $C$ of $\psat$, it is possible to construct a polycube $P_3^Q(C)$ of polynomial size in polynomial time.
\end{proposition}
\begin{proof}
    The proof follows the same strategy of Proposition~\ref{prop:psat-to-polyc-rooks} with different enlargements needed for the similar gadgets. It is thus omitted.
\end{proof}

We illustrate the process of Proposition~\ref{prop:psat-to-polyc-queens} on a small example by taking the instance $C$ of Figure~\ref{fig:ex_psat} and constructing the polycube $P^Q_3(C)$ associated with it in Figure~\ref{fig:ex_psat_poly}.

\begin{figure}
    \centering
    \begin{tikzpicture}[scale=.25]
    \foreach \x/\y in {
        1/1,2/1,3/1,4/1,5/1,6/1,7/1,8/1,9/1,10/1,11/1,12/1,
        1/2,2/2,3/2,4/2,5/2,6/2,7/2,8/2,9/2,10/2,11/2,12/2,
        3/3,4/3,10/3,11/3,
        3/4,4/4,5/4,6/4,7/4,8/4,10/4,11/4,
        3/5,4/5,5/5,6/5,7/5,8/5,10/5,11/5,
        3/6,4/6,10/6,11/6,
        10/7,11/7,
        1/8,2/8,3/8,4/8,5/8,6/8,7/8,8/8,10/8,11/8,
        1/9,2/9,3/9,4/9,5/9,6/9,7/9,8/9,10/9,11/9,
        3/10,4/10,10/10,11/10,
        3/11,4/11,10/11,11/11,
        10/12,11/12,
        3/13,4/13,10/13,11/13,
        3/14,4/14,
        1/15,2/15,3/15,4/15,5/15,6/15,7/15,8/15,9/15,10/15,11/15,
        1/16,2/16,3/16,4/16,5/16,6/16,7/16,8/16,9/16,10/16,11/16
        } {
                \foreach \z/\w in {0/0,-1/1,0/1,1/1,-2/2,-1/2,1/2,2/2,-1/3,0/3,1/3,0/4} {
                \path [draw=brown!, fill=brown!30] (.5+4*\x+\z-0.45, .5+4*\y+\w-1.45)
				   -- ++(0,.9)
				   -- ++(.9,0)
				   -- ++(0,-.9)
				   --cycle;
       }
       }
     \foreach \x/\y in {
        1/1,2/1,        1/2,2/2,
        1/8,2/8,        1/9,2/9,
        1/15,2/15,        1/16,2/16
        } {
                \foreach \z/\w in {0/0,-1/1,0/1,1/1,-2/2,-1/2,1/2,2/2,-1/3,0/3,1/3,0/4} {
                \path [draw=g-green!, fill=g-green!50] (.5+4*\x+\z-0.45, .5+4*\y+\w-1.45)
				   -- ++(0,.9)
				   -- ++(.9,0)
				   -- ++(0,-.9)
				   --cycle;
       }
       }
       \foreach \x/\y in {
        7/4,8/4,        7/5,8/5,
        7/8,8/8,        7/9,8/9
        } {
                \foreach \z/\w in {0/0,-1/1,0/1,1/1,-2/2,-1/2,1/2,2/2,-1/3,0/3,1/3,0/4} {
                \path [draw=le-pink!, fill=le-pink!30] (.5+4*\x+\z-0.45, .5+4*\y+\w-1.45)
				   -- ++(0,.9)
				   -- ++(.9,0)
				   -- ++(0,-.9)
				   --cycle;
       }
       }
       \foreach \x/\y in {
        3/10,4/10,        3/11,4/11,
        3/13,4/13,        3/14,4/14
        } {
                \foreach \z/\w in {0/0,-1/1,0/1,1/1,-2/2,-1/2,1/2,2/2,-1/3,0/3,1/3,0/4} {
                \path [draw=le-orange!, fill=le-orange!30] (.5+4*\x+\z-0.45, .5+4*\y+\w-1.45)
				   -- ++(0,.9)
				   -- ++(.9,0)
				   -- ++(0,-.9)
				   --cycle;
       }
       }
       \foreach \x/\y in {
        10/12,11/12,        10/13,11/13,
        10/15,11/15,        10/16,11/16
        } {
                \foreach \z/\w in {0/0,-1/1,0/1,1/1,-2/2,-1/2,1/2,2/2,-1/3,0/3,1/3,0/4} {
                \path [draw=le-purple!, fill=le-purple!30] (.5+4*\x+\z-0.45, .5+4*\y+\w-1.45)
				   -- ++(0,.9)
				   -- ++(.9,0)
				   -- ++(0,-.9)
				   --cycle;
       }
       }
       \draw (0,4*1+1) node {$x_3$};
       \draw (0,4*8-2) node {$x_2$};
       \draw (0,4*16-1) node {$x_1$};
       \draw (4*13,4*15+1) node[right,rotate=270] {$x_1  \lor \overline x_3$};
       \draw (4*6-1,4*13-2) node[rotate =270] {$\overline x_1 \lor \overline x_2$};
       \draw (4*8-2,4*7) node[rotate=90] {$  \overline x_3\lor \overline{x}_2$};
       \draw [decorate, decoration = {brace}] (-2,2) --  (-2,4*16+2);
       \draw (-6,4*8) node[rotate=90] {$x_1\lor x_2 \lor x_3$};
    \foreach \x/\y in {1/1,3/1,5/1,7/1,9/1,11/1,
    6/4,8/4,
    1/8,3/8,5/8,7/8,
     3/13,
     1/15,5/15,7/15,9/15,11/15
    }
    {
    \path [draw=red!, fill=red!90, opacity = .9] (.5+4*\x-0.45, .5+4*\y-1.45)
				   -- ++(0,.9)
				   -- ++(.9,0)
				   -- ++(0,-.9)
				   --cycle;
    }
     \foreach \x/\y in {2/1,4/1,6/1,8/1,10/1,12/1,
     5/4,7/4,
     2/8,4/8,6/8,8/8,
    4/13,
    2/15,6/15,8/15,10/15
     }
    {
    \path [draw=gray!, fill=gray!90, opacity = .9] (.5+4*\x-0.45, .5+4*\y-1.45)
				   -- ++(0,.9)
				   -- ++(.9,0)
				   -- ++(0,-.9)
				   --cycle;
    }
    \foreach \x/\y in {1/2,5/2,7/2,9/2,
    6/5,8/5,
    3/6,
    1/9,5/9,7/9,
    3/11,
    10/13,
    1/16,3/16,5/16,7/16,9/16,11/16}
    {
    \path [draw=gray!, fill=gray!90, opacity = .9] (.5+4*\x-0.45, .5+4*\y+4-1.45)
				   -- ++(0,.9)
				   -- ++(.9,0)
				   -- ++(0,-.9)
				   --cycle;
    }
     \foreach \x/\y in {2/2,6/2,8/2,12/2,
    5/5,7/5,
    4/6,
    2/9,6/9,8/9,
    4/11,
    11/13,
    2/16,4/16,6/16,8/16,10/16
    }
    {
    \path [draw=red!, fill=red!90, opacity = .9] (.5+4*\x-0.45, .5+4*\y+4-1.45)
				   -- ++(0,.9)
				   -- ++(.9,0)
				   -- ++(0,-.9)
				   --cycle;
    }
    \foreach \x/\y in {12/1,11/4,11/6,11/8,11/10,11/12,11/16,
    4/3,4/10,4/13,
    8/5,8/8}
    {
    \path [draw=red!, fill=red!90, opacity = .9] (.5+4*\x+2-0.45, .5+4*\y+2-1.45)
				   -- ++(0,.9)
				   -- ++(.9,0)
				   -- ++(0,-.9)
				   --cycle;
    }
     \foreach \x/\y in {4/6,4/11,4/14,
     8/4,8/9,
     11/3,11/5,11/7,11/9,11/11,11/13,11/15,
     12/2}
    {
    \path [draw=gray!, fill=gray!90, opacity = .9] (.5+4*\x+2-0.45, .5+4*\y+2-1.45)
				   -- ++(0,.9)
				   -- ++(.9,0)
				   -- ++(0,-.9)
				   --cycle;
    }
    \foreach \x/\y in {1/2,1/9,1/16,
    3/4,3/6,3/11,3/14,
    10/3,10/5,10/7,10/9,10/11,10/13}
    {
    \path [draw=red!, fill=red!90, opacity = .9] (.5+4*\x-2-0.45, .5+4*\y+2-1.45)
				   -- ++(0,.9)
				   -- ++(.9,0)
				   -- ++(0,-.9)
				   --cycle;
    }
     \foreach \x/\y in {1/1,1/8,1/15,
     3/3,3/5,3/10,3/13,
     10/4,10/6,10/8,10/10,10/12}
    {
    \path [draw=gray!, fill=gray!90, opacity = .9] (.5+4*\x-2-0.45, .5+4*\y+2-1.45)
				   -- ++(0,.9)
				   -- ++(.9,0)
				   -- ++(0,-.9)
				   --cycle;
    }
       \foreach \x/\y in {1/3,1/4,1/5,1/6,1/7,1/8,1/10,1/11,1/12,1/13,1/14,1/15, 
       3/11,3/12 
       }
       {\foreach \z/\w in {-2/1,-1/1,0/1,-2/2,-2/3,-1/3,0/3,-2/4} {
                \path [draw=gray!, fill=gray!30, opacity = .5] (.5+4*\x+\z-0.45, .5+4*\y+\w-1.45)
				   -- ++(0,.9)
				   -- ++(.9,0)
				   -- ++(0,-.9)
				   --cycle;
       }}
       \foreach \x/\y in {1/16, 
       3/13 
       }
       {\foreach \z/\w in {-2/1,-1/1,-0/1,-2/2,-3/2,-4/2} {
                \path [draw=gray!, fill=gray!30, opacity = .5] (.5+4*\x+\z-0.45, .5+4*\y+\w-1.45)
				   -- ++(0,.9)
				   -- ++(.9,0)
				   -- ++(0,-.9)
				   --cycle;
       }}
       \foreach \x/\y in {1/2, 
       3/10 
       }
       {\foreach \z/\w in {-2/3,-1/3,0/3,-2/2,-3/2,-4/2,-2/4} {
                \path [draw=gray!, fill=gray!30, opacity = .5] (.5+4*\x+\z-0.45, .5+4*\y+\w-1.45)
				   -- ++(0,.9)
				   -- ++(.9,0)
				   -- ++(0,-.9)
				   --cycle;
       }}
       \foreach \x/\y in {1/9 
       }
       {\foreach \z/\w in {-2/1,-1/1,-0/1,-2/2,-3/2,-4/2,-2/3,-1/3,-0/3,-2/4} {
                \path [draw=gray!, fill=gray!30, opacity = .5] (.5+4*\x+\z-0.45, .5+4*\y+\w-1.45)
				   -- ++(0,.9)
				   -- ++(.9,0)
				   -- ++(0,-.9)
				   --cycle;
       }}
       \foreach \x/\y in {
      8/5,8/6,8/7,8/8, 
       11/14,11/15 
       }
       {\foreach \z/\w in {2/1,1/1,0/1,2/2,2/3,1/3,-0/3,2/4} {
                \path [draw=gray!, fill=gray!30, opacity = .5] (.5+4*\x+\z-0.45, .5+4*\y+\w-1.45)
				   -- ++(0,.9)
				   -- ++(.9,0)
				   -- ++(0,-.9)
				   --cycle;
       }}
       \foreach \x/\y in {
       8/9, 
       11/16 
       }
       {\foreach \z/\w in {2/1,1/1,0/1,2/2,3/2,4/2} {
                \path [draw=gray!, fill=gray!30, opacity = .5] (.5+4*\x+\z-0.45, .5+4*\y+\w-1.45)
				   -- ++(0,.9)
				   -- ++(.9,0)
				   -- ++(0,-.9)
				   --cycle;
       }}
       \foreach \x/\y in {8/4, 
       11/13 
       }
       {\foreach \z/\w in {2/3,1/3,0/3,2/2,3/2,4/2,2/4} {
                \path [draw=gray!, fill=gray!30, opacity = .5] (.5+4*\x+\z-0.45, .5+4*\y+\w-1.45)
				   -- ++(0,.9)
				   -- ++(.9,0)
				   -- ++(0,-.9)
				   --cycle;
       }}
    \end{tikzpicture}
    \caption{Top view of the polycube $P^Q_3(C)$ coming from the instance $C$ of Figure~\ref{fig:ex_psat}. The variables $x_1,x_2,x_3$ are in teal
    and propagated by the brown polycubes. The clauses are in gray and are \emph{on top} of the colored polycubes. The tiles in red and dark gray represent some F and T cubes, respectively. The maximum number of queens lies between 380 and 384, as $n_{2neigh} = 11$, $n_{3neigh} = 38$, $n_{4neigh} = 5$, and the length of the clauses is given by $\ell_Q(c_1) = 57$, $\ell_Q(c_2) = 13$, $\ell_Q(c_3) = 21$, $\ell_Q(c_4) = 13$. There are 6 maximum placements  of 384 queens, 2 for each of the 3 solutions $(1,0,1)$, $(1,0,0)$ and $(0,1,0)$}.
    \label{fig:ex_psat_poly}
\end{figure}

We now show how to translate the instance of $\psat$ into a domination problem.
\begin{lemma}\label{lem:equivalent-queens}
    Let $C$ be an instance of $\psat$ and $P_3^Q(C)$ be its associated polycube. Let
    \[
    m_Q := 4n_{4neigh} + 5 n_{3neigh} + 6 n_{2neigh}+ \sum_{c\in \{\text{clauses of }C\}} \ell_Q(c),
    \]
    where $n_{Ineigh}$ is the number of pairs of elements with $I$ neighbors and $\ell_Q$ is the length of the queen clause gadget. Then one can place $m_Q+n_{clause}$ non-attacking queens on $P_3^Q(C)$ if and only if $C$ is satisfiable.
\end{lemma}
\begin{proof}
    Let $C$ be an instance, and construct the polycube $P_3^Q(C)$. If $C$ is satisfiable,  there is an assignment of values ensuring that all clauses are satisfied. For each satisfied clause, there will be a queen added to the clause gadget. Then, from Lemmas~\ref{lem:literal-gadget-queens_3D}--\ref{lem:clause-gadgets-queens}, we get that there are $m_Q+n_{clause}$ queens.

    If the polycube $P_3^Q(C)$ has a maximum placement of $m_Q+n_{clause}$, it means that all clauses are satisfied, since they are the only way to add more than $n_{Q}$ queens. Thus, $C$ is satisfiable.
\end{proof}

With this last lemma, everything is in place to prove Theorem~\ref{thm:nphard-queens}.

\begin{proof}[Proof of Theorem~\ref{thm:nphard-queens}]
Lemma~\ref{lem:verifP} shows that verifying a solution is done in polynomial time, and thus that the problem is in the class NP.

Let $C$ be an instance of $\psat$ and $P_3^Q(C)$ be the polycube constructed from it. From Proposition~\ref{prop:psat-to-polyc-queens}, we know $P_3^Q(C)$ is polynomially-sized and was constructed in polynomial time. Lemma~\ref{lem:equivalent-queens} then shows that it is equivalent to finding a guarding set of $m_Q+n_{clause}$ queens. As $\psat$ is NP-complete by Proposition~\ref{prop:psatNP}, this means that the non-attacking queen set problem for $3$--polycubes is NP-complete, and then it is also NP-complete for $d$-polycubes when $d\geq 3$.
\end{proof}

\section{Investigation on the domination of polyominoes}\label{sec:openquestions}
In this section, we consider open questions for polyominoes, specifically for the class of convex polyominoes. The interested reader can use Tool~\ref{tool:gadget-check}  to get the number of maximum dominating queen positions on any gadget on a polyomino. The underlying algorithm is explained in the next section.

\subsection{Maximum queen domination on polyominoes}
For the maximum rook domination problem, Theorem~\ref{thm:nphard-rooks} and~\cite[Thm~12]{alpert2021art} present the whole picture: there is a polynomial algorithm to solve the instance on polyominoes, and it is NP-complete for higher-dimensional polycubes. The argument of~\cite[Thm~12]{alpert2021art} translates this into a bipartite graph matching problem for which known polynomial algorithms exist. We present another argument that it is in P and extend this discussion at the end of Section~\ref{sec:modelcompu}.

In the case of maximum queen domination, Theorem~\ref{thm:nphard-queens} leaves the question of whether maximum queen domination on polyominoes is NP-complete as in the higher-dimensional case or in P as the rook domination on polyominoes unanswered. Using a similar approach to the rook case proves ineffective: the problem translates into a 4-hypergraph matching problem, which is NP-complete~\cite{karp1972}, thus leaving us unable to conclude. However, we strongly suspect the problem to be NP-complete and put forward the following conjecture.

\begin{conjecture}\label{conj:nphard-queens-2d}
    The maximum independent queen domination problem on polyominoes is NP-complete.
\end{conjecture}

As a support to this conjecture, we first consider a slight generalization of the problem. This generalization is equivalent to what is considered by Martin~\cite{barnaby2007}, and we provide another proof using $n$-queens completion.
\begin{definition}
    A \textbf{path-connected polyomino} is a finite collection of unit squares connected by their edges or by their vertices.
\end{definition}
The NP-completeness of the maximum independent queen domination on path-connected polyominoes was proven by Martin by proving the NP-completeness of maximum independent queen domination on ``walled chessboards.''
\begin{theorem}[{\cite[Thm 1]{barnaby2007}}]\label{thm:pathpolyo}
   Maximum independent queen domination is NP-complete on path-connected polyominoes.
\end{theorem}
\begin{proof}
A walled chessboard, according to Martin, is a chessboard with a set of tiles that stops the queen's ray of attack. This is an equivalent definition to that of a path-connected polyomino.
\end{proof}

Of course, this does not say anything about the complexity of the instance on the smaller subset of polyominoes, but along with extensive computations, it encourages us to believe the conjecture. We think this problem might be an interesting candidate to find a manageable reduction.

\subsection{Domination problems on convex polyominoes}
\begin{definition}[Convex polyominoes]
Let $P$ be a polyomino. A polyomino is called \textbf{row-convex} if each row of $P$ has at most one connected component, it is called \textbf{column-convex} if each column of $P$ has at most one connected component, and it is called \textbf{convex} if it is both row- and column-convex.
\end{definition}

\begin{figure}[h]
    \begin{center}
    \begin{tikzpicture}
        \foreach \x/\y in {1/1, 2/1, 3/1, -1/2, 0/2, 1/2, 2/2, 3/2, 4/2, -1/3, 0/3, 1/3, 2/3, 3/3, 4/3, 5/3, -2/4, -1/4, 0/4, 1/4, 2/4, 3/4, 4/4, -1/5, 0/5, 1/5, 2/5, 3/5, 4/5, 1/6} {
            \path [draw=brown!70, fill=brown!70] (-8+3.5+\x-0.45, 1.5+\y-1.45)
            -- ++(0,.9)
            -- ++(.9,0)
            -- ++(0,-.9)
            --cycle;
        }

        \foreach \x/\y in {} {
            \node[anchor=west] at (-8+3.5+\x-0.4, \y+.5) {\rook};
        }
    \end{tikzpicture}
    \end{center}
    \caption{A convex polyomino.}\label{fig:convex-min}
\end{figure}

The problem of minimum domination for rooks on $n\times n$ square polyominoes is trivial: $n$ rooks are needed and simply putting them on the diagonal gives a solution. However, the same problem for polyominoes is NP-complete (Theorem~\ref{thm:minrookNPcomplete}). Convex polyominoes somehow lie in the middle of these two classes of polyominoes.

It is tempting to think that a simple process mimicking the diagonal domination on $n\times n$ chessboards would work for convex polyominoes; however, there are convex polyominoes where such processes would never find the optimal solution, such as the one presented in Figure~\ref{fig:pathological-polyo}. In the next section, we will present the general solver that we used to generate the solution.  We therefore conjecture that the minimum domination problem for rooks on convex polyominoes is NP-complete.

\begin{conjecture}\label{conj:convex-rooks}
The minimum domination problem for attacking or non-attacking rooks on convex polyominoes is NP-complete.
\end{conjecture}
Note that~\cite[Lem.~14]{alpert2021art} implies that proving the conjecture for attacking rooks is equivalent to proving the non-attacking case.

\begin{figure}[h]
  \centering
  \begin{tikzpicture}
  \foreach \x in {1,...,9}
   {\foreach \y in {1,...,9}
   {  \path [draw=brown!70, fill=brown!70] (2.5+\x-0.45, 0.5+\y-.45)
              -- ++(0,.9)
              -- ++(.9,0)
              -- ++(0,-.9)
              --cycle;
              }
              }
   \foreach \x/\y in {0/0,0/1,0/2,0/3,
   1/0,1/1,1/2,1/3,1/4,1/5,
   2/0,2/1,2/2,2/3,2/4,2/5,2/6,2/7,
   3/0,4/0,4/10,5/0,5/10,6/0,6/10,7/0,7/10,8/0,8/10,9/0,9/10,
   12/1,12/2,12/3,12/4,12/5,12/6,12/7,12/8,13/1,13/2} {
              \path [draw=brown!70, fill=brown!70] (.5+\x-0.45, .5+\y-.45)
              -- ++(0,.9)
              -- ++(.9,0)
              -- ++(0,-.9)
              --cycle;
          }

          \foreach \x/\y in {3/1,4/5,5/2,6/8,7/4,8/0,9/3,10/7,11/6} {
              \node[anchor=west] at (.5+\x-0.45, \y+.5) {\rook};
          }
  \end{tikzpicture}
  \caption{A convex polyomino with minimum domination number 9.}\label{fig:pathological-polyo}
\end{figure}

 We also conjecture the same complexity for the minimum domination problems of queens on convex polyominoes.
\begin{conjecture}\label{quest:convex-min-dom-queens}
The minimum domination problem for attacking, or non-attacking, queens on convex polyominoes is NP-complete.
\end{conjecture}

On the other hand, as a consequence of~\cite[Thm 12]{alpert2021art}, which states that finding a maximum independent dominating set for rooks on polyominoes is in P, we get that the maximum independent domination problem for rooks on convex polyominoes is also in P.

\begin{corollary}\label{coro:Pconvex-max}
Maximum domination by non-attacking rooks on convex polyominoes is in P.
\end{corollary}
\begin{proof}
Since it is in P for all polyominoes~\cite[Thm 12]{alpert2021art}, it is also in P for the convex ones.
\end{proof}

The maximum domination problem on convex polyominoes for non-attacking queens, on the other hand, has yet to be studied.  The gadgets used in the proofs on ~\cite{alpert2021art} are non-convex; thus, a direct proof cannot come from this method. However, the proof of the NP-completeness of the completion for the chessboard in~\cite{gent2017complexity} has gadgets that are non-convex only due to the constraint of completing the square. We suspect that an adaptation of their proof could lead to a proof of the result for convex polyominoes, though the reduction could be very subtle.
\begin{conjecture}\label{quest:convex-max-dom-queens}
The maximum domination problem for non-attacking queens on convex polyominoes is NP-complete.
\end{conjecture}

One subfamily of convex polyominoes is the square polyominoes. There is no known polynomial time algorithm for finding the minimum number of queens needed to guard a $n\times n$ chessboard. We define the minimum chessboard domination queen completion problem as $(n, Q, l)$. Given a set $Q$ of $k$ queens placed on the $n \times n$ chessboard, is there a set $Q^\prime$ of size $k + l$ such that $Q \subset Q^\prime$ and that it dominates (guards) the chessboard?

\begin{conjecture}\label{conj:min-dom-queens}
The minimum chessboard domination queen completion problem is NP-complete.
\end{conjecture}

 Analogously, we define the minimum independent chessboard domination queen completion problem $(n, Q, l)^I$. Given a set $Q$ of $k$ non-attacking queens placed on the $n \times n$ chessboard, is there a set $Q^\prime$ of size $k + l$ such that $Q \subset Q^\prime$, that it guards the chessboard, and no pair of queens in $Q^\prime$ attack each other? Notice that the difference with the well-known $n$-queens completion problem is that with $n$ queens the $n \times n$ chessboard is trivially dominated. As it has already been proven that this problem is NP-complete ~\cite[Thm~37]{gent2017complexity}, we can prove the following corollary.

\begin{corollary}\label{coro:min-ind-dom-npcomp-queens}
The minimum independent chessboard domination queen completion problem is NP-complete.
\end{corollary}
\begin{proof}

Let $n\in \mathbb{N}$ and $0\leq k\leq n$. An instance $(n,Q)$ of the $n$--queens completion problem can be reduced in polynomial time to an instance of the minimum independent chessboard domination queen completion problem $(n,Q,n-k)^I$. Since $n$--queens completion is NP-complete~\cite[Thm~37]{gent2017complexity}, it is NP-hard and the minimum independent domination queen completion problem is NP-hard. Finally, verifying a solution is done in polynomial time by Lemma~\ref{lem:verifP}, so the proof is completed.
\end{proof}

\section{An integer linear programming solver}\label{sec:modelcompu}
In this section, we present a general solver using integer linear programming (ILP) for the domination problem. This allows us to make use of efficient solvers such as~\cite{huangfu_parallelizing_2018}. We illustrate the performance of this general scheme by comparing it against heavily optimized solvers~\cite{bird_2017} created specifically for the queen domination problem on $n\times n$ chessboards.

Given a polycube, let $m$ be its number of tiles. We start by enumerating the tiles of the polycube in arbitrary order with an index $i \in \{1, \ldots, m\}$. We introduce $m$ binary variables $x_i$, with $x_i = 1$ encoding that a chess piece is placed on this tile and $x_i = 0$ encoding that no chess piece is placed on this tile.

Let us use $v$ to denote the attack direction of a selected chess piece $F$. For example, for a queen on the tile $i$ on a given polyomino there are four attacking directions: (1) the row it is on; (2) the column (see Figure~\ref{fig:col-attack-direction}); (3) the ascending diagonal; and (4) the descending diagonal. For a three-dimensional queen, there will be 13 attack directions.
\begin{figure}[h]
\centering
\begin{tikzpicture}[scale=1]
\foreach \x/\y in {-1/-1, 0/-1, 1/-1, 2/-1, 2/0, 0/1, -1/1, -2/1, 2/1, 0/2, 1/2, 2/2, -2/2, 0/3, -1/3, -2/3, 0/4, 1/4, -1/4, -2/4, 0/5, 1/5, 2/5} {
\path [draw=black, fill=tr-blue, fill opacity=0.6] (.5+\x-0.45, .5+\y-0.45)
-- ++(0,.9)
-- ++(.9,0)
-- ++(0,-.9)
--cycle;
}

\foreach \x/\y in {0/1,  0/2,  0/3,  0/4,  0/5}
{
\path [draw=black, fill=tr-pink] (.5+\x-0.45, .5+\y-0.45)
-- ++(0,.9)
-- ++(.9,0)
-- ++(0,-.9)
--cycle;
}
\draw (0.5, 3.5) node {$i$};
\end{tikzpicture}
\caption{A polyomino with the column attacking direction of tile $i$ highlighted in pink.}\label{fig:col-attack-direction}
\end{figure}

For each tile $i$ of the polycube and each attack direction $v$ of the piece $F$, we construct a set $I^F_{v, i}$ consisting of all the labels of the tiles that can be attacked from the tile $i$ using only the attack direction $v$; note that $i\in I^F_{v,i}$. Then, we add the following constraint to our integer linear programming problem for each tile $i$ and each attack direction $v$:
\begin{equation*}
    \sum_{j \in I^F_{v, i}} x_j \leq 1.
\end{equation*}

This ensures that at most one of the $x_j$  for $j\in I^F_{v,i}$ can be $1$, meaning that for all pairs of tiles attacking each other, there is only one chess piece placed on one of them. This construction leads in many cases to duplicate constraints, which we can simply remove at the end to optimize the program.

If we now maximize with the objective function $\sum_{i = 1}^m x_i$, which is the number of chess pieces, we have a translation of the maximum domination problem for non-attacking rooks or queens into an integer linear programming problem. This is one of the classical approaches to solve chessboard domination problems and is described in more detail in~\cite{foulds_application_1984}.

To translate the minimum independent domination problem, we need additional constraints to guarantee the domination of the polycube. For each tile of the polycube, we construct a set $A_i^F$ consisting of the tile $i$ and all tiles that can attack the tile $i$ using one movement of the piece $F$; note that $A_i^R\subset A^Q_i$. We add the following constraint to our integer linear programming problem for each tile $i$:
\begin{equation*}
    \sum_{j \in A_i^F} x_j \geq 1.
\end{equation*}
This ensures that each tile is guarded, since for each tile at least one of the $x_j$ in $A_i^F$ must be $1$, meaning that a chess piece is either placed on the tile or one can attack the tile.

Combining the two sets of constraints, we get the minimum independent domination integer linear programming problem for queens or rooks.
\begin{equation}\label{eq:ILP}
\begin{aligned}
(\mathrm{ILP}) \quad
& \text{minimize}   &&  \sum_{i = 1}^m x_i\\
& \text{subject to} && \mathbf{x} \in \{0,1\}^m, \\
& \quad\text{(Independence)} &&  \sum_{j \in I^F_{v, i}} x_j \leq 1,\quad \text {for all } I^F_{v,i}\\
&\quad \text{(Domination)} &&\sum_{j \in A_i^F} x_j \geq 1, \quad \text{for all } A_i^F.
\end{aligned}
\end{equation}

\begin{tool}\label{tool:gadget-check}
As a tool to help further research, we designed a small helper program based on the Julia library. It uses the exhaustive search provided by the proprietary solver Gurobi on the previously discussed ILP model to find all possible combinations of placing $d$ non-attacking queens on a polyomino~\cite{GadgetCheck}.
\end{tool}

The translation to an ILP process was used to calculate the previously unknown minimum numbers of non-attacking queens guarding the $n\times n$ chessboards for $n=26,\ldots,31$, which allowed us to extend the OEIS sequence \href{https://oeis.org/A075324}{A075324}~\cite{oeis}. To model the general minimum domination problem (sequence \href{https://oeis.org/A075458}{A075458}), one only needs to remove the independence constraints.

Our model can also be translated into a boolean satisfiability problem, since $\sum_i x_i \geq 1$ is equivalent to $\bigvee_i x_i$ and $\sum_i x_i \leq 1$ is equivalent to $\bigwedge_{i} \bigwedge_{j, i \neq q} (\overline{x}_i \lor \overline{x}_j$). This allowed us to solve the optimization problems with the state-of-the-art maxSAT solver Loandra~\cite{loandra}, but its runtime could not compete with the ILP solvers we tested. Interestingly, the better performance of ILP for these types of problems was commented upon by Knuth\footnote{``The author was able to reach $n = 47$ with dancing-links-based methods, in an
afternoon. But he knew that integer programming is significantly faster for `linear' applications such as the $n$ queens problem (see answer 36). So he enlisted the help of Matteo Fischetti; and sure enough, Matteo was able to extend the results dramatically''~\cite[Answer~236, p.477]{KnuthTAOCP4b}.}; see the analysis of Fischetti and Salvagnin on a generalization of the $n$-queens puzzle for an example~\cite{Fischetti18}.

Furthermore, we tested the performance of our solver on the problem of finding the maximum number of non-attacking queens that can be placed on a $d$--hypercube of size $n$.

\begin{table}[h]
    \centering{}%
    \begin{tabular}{l|cccccccccccccc}
    $8$ & $1$ & $1$ & $\color{le-orange} 52$\\
    $7$ & $1$ & $1$ & $\color{le-orange} 32$ & $\color{le-red} \mathbf{128}$\\
    $6$ & $1$ & $1$ & $\color{le-orange} 19$ & $\color{le-red} \textbf{64}$\\
    $5$ & $1$ & $1$ & $\color{le-orange} 11$ & $\color{le-red} \mathbf{32}$\\
    $4$ & $\color{g-green} 1$ & $\color{g-green}1$ & $\color{g-green} 6$ & $\color{g-green} 16$ & $\color{le-red} \mathbf{38}$ & $\color{le-red} \mathbf{80}$ & $\color{le-red} \mathbf{145}$\\
    $3$ & $\color{g-blue} 1$ & $\color{g-blue} 1$ & $\color{g-blue} 4$ & $\color{g-blue} 7$ & $\color{g-blue} 13$ & $\color{g-blue} 21$ & $\color{g-blue} 32$ & $\color{g-blue} 48$ & $\color{g-blue} 67$ & $\color{g-blue} 91$ & $\color{g-blue} 121$ & $\color{g-blue} 133$ & $\color{g-blue} 169$ \\
    $2$ & $1$ & $1$ & $\color{le-orange} 2$ & $4$ & $5$ & $6$ & $7$ & $8$ & $9$ & $10$ & $11$ & $12$ & $13$ \\
    \hline
     $\sfrac{d}{n}$ & $1$ & $2$ & $3$ & $4$ & $5$ & $6$ & $7$ & $8$ & $9$ & $10$ & $11$ & $12$ & $13$
    \end{tabular}
    \caption{Values of maximum queen domination problem on a $n^d$ hypercube. In blue (line $d=3$), on a cube; in green (line $d=4)$,  on a tesseract, and in orange (column $n=3$), on a $d$--dimensional hypercube of side length 3. The red numbers in bold are previously unknown values that we have calculated using \cite{GitHubHypercube}.}
    \label{tab:known-values}
\end{table}

Table~\ref{tab:known-values} presents  for reference many sequences of maximum independent queen domination problems on a $n^d$ hypercube. Some of the sequences are known and are indexed in OEIS~~\cite{oeis}: the queens on a cube of size $n$~ \href{https://oeis.org/A068940}{A068940}; queens on a tesseract \href{https://oeis.org/A068941}{A068941}; and queens on a $d$--dimensional hypercube of side length $3$ \href{https://oeis.org/A115992}{A115992}. We have extended the sequences of domination numbers for $d$--dimensional tesseracts of side length $4$ and of $4$--dimensional tesseracts of side length $n$  using our model---the program is available at~\cite{GitHubHypercube}.

From the data we found with our computations, we conjecture the following formula for the maximum number of non-attacking queens on the hypercube.
\begin{conjecture}\label{conj:queens-hyper}
The maximum non-attacking queen set problem on a d-dimensional hypercube of side length $4$ has the solution size $2^d$ for $d \geq 4$.
\end{conjecture}

To complement the exact algorithm, we now comment on the possibilities to approximate the domination problems we consider. First, the matrix of the rook maximum domination problem is unimodular, providing an alternate proof or \cite[Thm~12]{alpert2021art}.  Yet another proof is provided by noting that the rook graph is claw free, so finding an independent set is polynomial~\cite{minty1980maximal}.

For queens, however, or even rooks on a polycube, the graph is not claw free. However, we can still say something by extending the notion. A $m$-claw is the complete bipartite graph $K_{1,m}$ shown in Figure~\ref{fig:dclaw}. The normal claw is the bipartite graph $K_{1,3}$.
\begin{figure}[h!]
\[
\left.\begin{tikzpicture}[baseline= {(current bounding box.center)}]
    \fill (0,0) circle(2pt);
    \fill (2,2) circle(2pt);
    \fill (2,1) circle(2pt);
    \fill (2,0) circle(2pt);
    \draw (2,-1) node {$\vdots$};
    \fill (2,-2) circle(2pt);
    \draw (0,0) -- (2,2);
    \draw (0,0) -- (2,1);
    \draw (0,0) -- (2,0);
    \draw (0,0) -- (2,-2);
\end{tikzpicture}\right\} m
\]
\caption{A $m$--claw.}\label{fig:dclaw}
\end{figure}
We say that a connected graph is $m$--claw free if it does not contain the $m$--claw as an induced subgraph. We construct a chess graph on a polycube associated with a piece $p$ by putting a vertex for each unit polycube and an edge between two vertices if the piece $p$ can travel from one to another. Chess graphs are $m$--claw free for a certain $m$ related to the piece $p$.
\begin{proposition}\label{prop:Chess_Graph_claw_free}
Let $p$ be a chess piece with $m$ attack directions. Then, the associated chess graph is $(m+1)$--claw free.
\end{proposition}
\begin{proof}
    Suppose there is a $(m+1)$--claw as a subgraph. This means that a piece can reach $m+1$ independent tiles. As it has $m$ attack direction, by the pigeonhole principle, two of the tiles are in the same attack direction and there must be an edge between them.
\end{proof}
We remark briefly that the converse of Proposition~\ref{prop:Chess_Graph_claw_free} is not true as can be seen by the following example of a claw-free graph and a 4--claw-free graph that do not correspond to the rook and queen domination problems, respectively.
\begin{figure}[h!]
\[
\begin{tikzpicture}[scale=2]
\draw (0,0) -- (1,0) -- (1,1) -- (0,1)--cycle;
\fill (0,0) circle (2pt);
\fill (0,1) circle(2pt);
\fill (1,0) circle(2pt);
\fill (1,1) circle(2pt);
\end{tikzpicture}
\qquad
\begin{tikzpicture}[scale=2]
\draw (0,0) -- (1,0) -- (1,1) -- (0,1)--cycle;
\draw (0,0)--(1,1);
\fill (0,0) circle(2pt);
\fill (0,1) circle(2pt);
\fill (1,0) circle(2pt);
\fill (1,1) circle(2pt);
\end{tikzpicture}
\]
\caption{A 4-claw-free graph and a claw-free graph. There are no tetraminoes with a corresponding queen graph for the first graph, and no tetraminoes with a corresponding rook graph for the second graph.}
\end{figure}

Independent set is in P for claw-free graphs. For $m$-claw-free graphs with $m>3$, there is a constant factor approximation algorithm.
\begin{theorem}[{\cite[Thm~13]{Neu21}}]\label{thm:dClaw-free_poly_approx}
Let $G$ be a $m$--claw-free graph. Then, an independent set of $G$ can be  approximated in polynomial time with factor $m/2$.
\end{theorem}

The class of $m$--claw-free graphs has an interesting relation between the minimum independent domination and the maximum independent domination problems.
\begin{proposition}\label{prop:Min_and_Max_related}
Let $G$ be a $m$--claw-free graph. Then, if we denote by $\min(G)$ the minimum size of an independent dominating set on $G$ and by $\max(G)$ the maximum size of an independent dominating on $G$, we have
 \begin{equation}
     (m-1)\operatorname{min}(G) \geq \operatorname{max}(G).
 \end{equation}
\end{proposition}
\begin{proof}
    Let $M$ be a minimum independent dominating set on $G$. Then, each vertex in $M$ divides into at most $m-1$ fully connected sets. There are thus at most $(m-1)\operatorname{min}(G)$ disjoint induced fully connected subgraphs covering the whole graph.
\end{proof}
As a corollary, we thus have a constant-factor approximation algorithm for the maximum independent domination problem and an upper bound for the minimum independent domination problem.
\begin{corollary}\label{coro:min_constant_factor}
The maximum independent domination on a $m$--claw-free graph can be approximated polynomially by a factor of $m/2$. Minimum independent domination can be bounded in polynomial time by a constant factor $m(m-1)/2$. In particular, this applies to chess graphs.
\end{corollary}
\begin{proof}
    This is a consequence of Theorem~\ref{thm:dClaw-free_poly_approx} and Proposition~\ref{prop:Min_and_Max_related}. The last statement follows directly from Proposition~\ref{prop:Chess_Graph_claw_free}.
\end{proof}

This last corollary gives the constant factor approximation and bound for any given dimension. However, the constant factor grows with the dimension. It is linear for the rooks, as a $d$-dimensional rook has $d$ attack directions, and it is exponential for the queens, as the number of attack directions of a $d$-dimensional queen is $(3^d-1)/2$. To see this last claim, place a queen at the middle of a $d$-hypercube of side 3: it can go in one move in all the $3^d-1$ remaining unit cubes and each cube as a cube opposite in a straight line; see Figure~\ref{fig:mov-pieces} for the example in dimension 3.

To conclude, we now describe the video game we created with the Godot game engine~\cite{godot} on minimum rook and queen domination on polyominoes. Readers are invited to try the game online  at~\url{https://www.erikaroldan.net/queensrooksdomination}.

The game challenges the player to dominate a random polyomino either with rooks or queens. When the player submits their solution, the smallest number necessary is then given to them. The polyominoes chosen for the game are all 50-tile polyominoes generated via a shuffling algorithm with a percolation parameter empirically chosen to offer interesting challenges to the player. The optimal solutions for the generated polyominoes were found using the solver described above, which can in fact easily manage polyominoes with thousands of tiles.

We plan to further develop the game to include minimum independent domination and maximum independent domination, as well as create a two-player game and port it to different platforms.

\section*{Acknowledgements}
{\small Significant progress was made on this project while ALR and MM visited ERR at the Max Planck Institute for Mathematics in the Sciences. The three authors also worked together on this project at the conference \textit{Let’s talk about outreach!} that took part at SwissMAP Research Station, Les Diablerets, Switzerland.
 We would like to thank Jonas Handwerker for comments and suggestions. The third author thanks the Laboratory for Topology and Neuroscience at the EPFL for hosting them during part of the project. Finally, the authors would like to thank the anonymous referee for their careful reading and helpful suggestions.}

\section*{Declarations}
\subsection*{Ethical Approval} Not applicable.
\subsection*{Competing interests} The authors declare no conflict of interest related to this work.
\subsection*{Authors' contribution} ALR, MM and ERR took an active role in all parts of the research. ALR, MM and ERR wrote and reviewed the whole manuscript.
\subsection*{Funding}  Open Access funding enabled and organized by Projekt DEAL. ALR was supported in part by the EOS Research Project [grant number 30889451] and by a scholarship from the Fonds de recherche du Qu\'ebec -- Nature et technologies [grant number 270527]. This project received funding from the European Union's Horizon 2020 research and innovation program under the Marie Sk\l odowska-Curie [grant agreement No.~754462]. We acknowledge the support given under Federal Ministry of Education and Research of Germany and by Sächsische Staatsministerium für Wissenschaft, Kultur und Tourismus in the programme Center of Excellence for AI-research „Center for Scalable Data Analytics and Artificial Intelligence Dresden/Leipzig“ (project identification number: ScaDS.AI)
\subsection*{Availability of data and software}
The software developed and implemented in the course of this research is publicly available on GitHub~\cite{QandRsoftware},  the polyomino verification tool on~\cite{GadgetCheck}, and the program for the computations done in the last section is publicly available on GitHub~\cite{GitHubHypercube}.

\Urlmuskip=0mu plus 1mu\relax 

\printbibliography
\end{document}